\documentclass[12pt]{amsart}

\usepackage{amsmath,amssymb,txfonts}

\allowdisplaybreaks
\usepackage{times}
\usepackage{bbm}
\usepackage[dvips]{graphicx}
\allowdisplaybreaks
\usepackage{amsxtra}
\usepackage[cp1252]{inputenc}
\usepackage{color}

\usepackage{mathrsfs}

\numberwithin{equation}{section}

\newtheorem{Theorem}{Theorem}[section]
\newtheorem{thm}{Theorem}[section]
\newtheorem{theorem}{Theorem}[section]

\newtheorem{pro}[theorem]{Proposition}
\newtheorem{defi}{{Definition}}[section]

\newtheorem{lem}[theorem]{Lemma}

\newcommand{\ct}[1]{\langle {#1}\rangle \lower.3ex\mathrm{$_{t}$}}
\newcommand{\lt}[1]{[ {#1}] \lower.3ex\mathrm{$_{t}$}}

 \usepackage[colorlinks, citecolor=blue,pagebackref,hypertexnames=false]{hyperref}

\numberwithin{equation}{section}

\begin{document}

\title{Besov Spaces, Schatten Classes and Weighted Versions of the Quantised Derivative}

\author{Zhenbing Gong}
\address{Zhenbing Gong, Department of Applied Mathematics, School of Mathematics and Physics,
 University of Science and Technology Beijing,
Beijing 100083, China}
\email{gongzhenb@126.com}

\author{Ji Li}
\address{Ji Li, School of Mathematical and Physical Sciences, Macquarie University, NSW, 2109, Australia}
\email{ji.li@mq.edu.au}

\author{Brett D. Wick}
\address{Brett D. Wick, Department of Mathematics\\
         Washington University - St. Louis\\
         St. Louis, MO 63130-4899 USA
         }
\email{wick@math.wustl.edu}

  \date{\today}

 \subjclass[2010]{47B10, 42B20, 43A85}
\keywords{Schatten class, commutator, Riesz transform, Besov space}

\maketitle
\bigbreak

%
%
%
%
%

\bigskip
\begin{center}
Dedicated to Professor Oleg Besov on the occasion of his 90th birthday,
\end{center}
\bigskip

\begin{abstract}
		In this paper, we establish the Schatten class and endpoint weak Schatten class estimates for the commutator of Riesz transforms
on weighted $L^2$ spaces. It provides a weighted version for the estimate of the quantised derivative introduced by Alain Connes and studied recently by
Lord--McDonald--Sukochev--Zanin.
	\end{abstract}
	
	\maketitle

\section{Introduction}\label{Introduction}
The commutator $[b,T]$ of the singular integral operator $T$ with a symbol $b$, which is defined by
\begin{align*}
[b,T]f(x)=b(x)Tf(x)-T(bf)(x),
\end{align*}
has played a vital role in harmonic analysis, complex analysis and partial differential equations. We refer to the fundamental work by Nehari \cite{N},  Calder\'{o}n \cite{C1965} and Coifman--Rochberg--Weiss \cite{CRW1976}. It has been extensively studied by many authors in different aspects with various applications. See for example \cite{CDW2019, CLMS, Hy, P2003}.

Besides the boundedness and compactness, the Schatten class estimates of the commutator have been an important topic, as it connects to non-commutative analysis.
For example, the commutator of Riesz transforms $[b,R_j]$, $j=1,\ldots,n$, links to the quantised derivative
$$
\bar{d} b:=i\left[\operatorname{sgn}(\mathcal{D}), 1 \otimes M_b\right]
$$
of Alain Connes introduced in \cite[Chapter $\mathrm{IV}$]{C1994}, where $M_b$ is the multiplication operator defined as $M_bf(x) = b(x)f(x)$. Details of these notation will be stated in the last section. This has been intensively studied in \cite{FLMSZ, FSZ2022, GG2017, LMSZ2017, RS1989, JW1982}. We note that in \cite{LMSZ2017} they implemented a new approach to prove that for $b\in L^\infty(\mathbb R^n)$, $\bar{d} b$ is in the weak Schatten class if and only if $b$ is in the Sobolev space.

 In \cite{LLW2022}, the authors have considered the Schatten class estimate of the commutator of Hilbert transform in the two-weight setting (along the line of \cite{B} and \cite{LL2022}) and made a fundamental first step. 
\begin{flushleft}
	\textbf{Theorem A.}  Let  $H$ be the Hilbert transform on $\mathbb{R}$, $\mu, \lambda\in A_2$ and set $\nu=\mu^{\frac{1}{2}}\lambda^{-\frac{1}{2}}$. Suppose $b\in VMO$, then commutator $[b,H]$ belong to $S^2(L^2_{\lambda}(\mathbb{R}),L^{2}_{\mu}(\mathbb{R}))$ if and only if $b\in B_\nu^2(\mathbb{R})$.
\end{flushleft}
 
As commented in  \cite{LLW2022}, the full version of Schatten class estimate of $[b,H]$ is not known, nor the commutator of Riesz transforms. In fact, even the one weight setting has not been characterized before. Thus, a natural problem is to consider the charactherization of the Schatten class $S^{p}\ (0<p<\infty)$ of commutator of Riesz transforms in one-weight setting in  higher dimensional spaces. Therefore, in this paper we will consider the Schatten--Lorentz membership of the commutators acting on weighted spaces $L^2(w)$ for $w$ in the Muckenhoupt $A_2$ class, which provides a weighted version of \cite{LMSZ2017}. The main approach is via dyadic harmonic analysis, the decomposition of the cubes via the median of the VMO function, and via the nearly weakly orthonormal sequences as in \cite{RS1989}.

To state our result, we first recall  the Schatten classes.
Let $\mathcal{G}_1$ and $\mathcal{G}_2$ be separable complex Hilbert spaces.
Suppose $T$ is a compact operator from $\mathcal{G}_1$ to  $\mathcal{G}_2$, let $T^*$ be the adjoint operator, it is clear that $|T|=(T^*T)^{\frac{1}{2}}$ is a compact, self-adjoint and non-negative operator from $\mathcal{G}_1$ to  $\mathcal{G}_2$. Let $\left(\psi_k\right)_k$ be an orthonormal basis for $\mathcal{G}_1$ consisting of eigenvectors of $|T|$, and let $s_k(T)$ be the eigenvalue corresponding to the eigenvector $\psi_k, k\in \mathbb{Z}^+$.
The numbers $s_1(T) \geq s_2(T) \geq \cdots \geq s_n(T) \geq \cdots \geq 0$, are called the singular values of $T$.
If $0<p<\infty$, $0<q\leq \infty$ and the sequence of singular values is $\ell^{p,q}$-summable (with respect to a weight), then $T$ is said to belong to the Schatten--Lorentz class $S^{p,q}\left(\mathcal{G}_1, \mathcal{G}_2\right)$.

That is,
$$\|T\|_{S^{p,q}\left(\mathcal{G}_1, \mathcal{G}_2\right)}=\bigg(\sum_{k\in \mathbb{Z}^{+}}\left(s_{k}\left(T\right)\right)^{q}\left(1+k\right)^{\frac{q}{p}-1}\bigg)^{1\over q},\quad q<\infty, $$and
$$\|T\|_{S^{p,\infty}\left(\mathcal{G}_1, \mathcal{G}_2\right)}=\sup_{k\in \mathbb{Z}^{+}}s_{k}\left(T\right)\left(1+k\right)^{\frac{1}{p}},\quad q=\infty.  $$
Clearly, $S^{p,p}\left(\mathcal{G}_1, \mathcal{G}_2\right)=S^{p}\left(\mathcal{G}_1, \mathcal{G}_2\right)$. Moreover, see for example \cite{P2003}, we also have,
\begin{align}
S^{p_1,q_{1}}\left(\mathcal{G}_1, \mathcal{G}_2\right)\subset S^{p_2,q_2}\left(\mathcal{G}_1, \mathcal{G}_2\right)\quad \text{if}\quad p_1<p_2,
\end{align}
\begin{align}
S^{p,q_{1}}\left(\mathcal{G}_1, \mathcal{G}_2\right)\subset S^{p,q_2}\left(\mathcal{G}_1, \mathcal{G}_2\right)\quad \text{if}\quad q_1<q_2.
\end{align}
If $\mathcal{G}_1 = \mathcal{G}_2 = \mathcal{G}$, we will simply write $S^{p,q}(\mathcal{G},\mathcal{G}) = S^{p,q}(\mathcal{G})$.


 Suppose $w\in A_{2}$, which will be defined in the next section. It's easy to find that $[b,R_{j}]$ is bounded and compact on $L^{2}(w)$, if and only if $b$ is in the $\mathrm{BMO}$ space and $\mathrm{VMO}$ space; see for example \cite{CDW2019}. We now consider the Besov space $B_{n/p}^{p,p}(\mathbb{R}^n)$, $0<p<\infty$, defined as the set of $b\in L^1_{loc}(\mathbb R^n)$ such that
 $$\|b\|_{B_{n/p}^{p,p}(\mathbb{R}^n)}:=\bigg(\int_{\mathbb{R}^n}\int_{\mathbb{R}^n}\frac{|b(x)-b(y)|^{p}}{|x-y|^{2n}}dydx\bigg)^{\frac{1}{p}}<\infty.$$

{ Note that $B_{n/p}^{p,p}(\mathbb{R}^n) \subset {\rm{VMO}}$. Thus, for $b\in B_{n/p}^{p,p}(\mathbb{R}^n) $, $[b,R_{j}]$ is bounded and compact.}
Our first result then provides a characterization of when the commutator is in the weighted Schatten class $S_p(L^2(w))$ in terms of membership of the symbol in the Besov space $B_{n/p}^{p,p}(\mathbb{R}^n)$:

\begin{thm}\label{Athm1}
Suppose { $n>1$}, $0<p<\infty$, $w\in A_{2}$ and $b\in {\rm VMO}(\mathbb{R}^n)$. Then for any $j=1,2,\cdots,n$, the commutator $[b,R_{j}]\in S^{p}(L^{2}(w))$ if and only if

{\rm{(1)}} $ b\in B_{n/p}^{p,p}(\mathbb{R}^n)$, if $n<p<\infty$, we have $\|b\|_{B_{n/p}^{p,p}(\mathbb{R}^n)}\approx\|[b,R_{j}]\|_{S^{p}(L^{2}(w))}$;

{\rm{(2)}} $b$ is a constant when $0<p\leq n$.
\end{thm}

In Theorem \ref{Athm1}, we note that there is a ``cut-off'' in the sense that the function space collapses to constants when $p$ is less than the critical index $p = n$ of the dimension. This suggests that at the endpoint $p=n$ there might be a more interesting phenomenon going on when one replaces membership in the Schatten-Lorentz space by its membership in weak-type versions.  This leads to the following result at the critical index.
\begin{thm}\label{Athm2}
Suppose { $n>1$}, $b\in {\rm VMO}(\mathbb{R}^n)$, $w\in A_{2}$. Then for any $j=1,2,\cdot\cdot\cdot,n$, the commutator $[b,R_{j}]\in S^{n,\infty}(L^{2}(w))$ if and only if $b\in \dot{W}^{1,n}(\mathbb{R}^n)$. More precisely,
$$\|b\|_{\dot{W}^{1,n}(\mathbb{R}^n)}\approx\|[b,R_{j}]\|_{S^{n,\infty}(L^{2}(w))}.$$
Here $\dot{W}^{1,n}(\mathbb{R}^n)$ is the homogeneous Sobolev space on $\mathbb R^n$ defined by $\dot{W}^{1,n}(\mathbb{R}^n)=\{ b\in (\mathcal S(\mathbb R^n))': \nabla b \in L^n(\mathbb R^n) \}$ with the seminorm
$\|b\|_{\dot{W}^{1,n}(\mathbb{R}^n)}=\|\nabla b \|_{ L^n(\mathbb R^n)}$.
\end{thm}

Once we have this, we provide a new application to the quantised derivative of Connes.  This yields the following result:

\begin{thm}
Suppose $n>1$,  $f\in {\rm VMO}(\mathbb{R}^n)$, $w\in A_{2}$. Then $\bar{d} f\in S^{n,\infty}(\mathbb{C}^N \otimes L^2\left(w\right))$ if and only if $f\in \dot{W}^{1,n}(\mathbb{R}^n)$. Moreover,
$$
 \|\bar{d} f\|_{S^{n,\infty}(\mathbb{C}^N \otimes L^2\left(w\right))}\approx \|f\|_{ \dot{W}^{1,n}(\mathbb{R}^n)}.
$$
\end{thm}
Details of the proof of the application and how it follows immediately from Theorem \ref{Athm2} are given in Section \ref{application}.  The remainder of this paper is organized as follows. Section \ref{Prelimineries} provides preliminary background information and notation.  The proof of Theorem \ref{Athm1} is started in Section \ref{Proof1} where the proof of (1) in Theorem \ref{Athm1} is given. In Section \ref{Proof2}, we give the proof of (2) Theorem \ref{Athm1}, and Section \ref{Proof3} provides the proof of Theorem \ref{Athm2}.{  In Section \ref{one dimensional}, we discuss the one-dimensional case.}

Throughout this paper, using $A\lesssim B$ and $A\gtrsim B$ to denote the statement that $A \leq CB$ and $A\geq CB$ for some constant $C > 0$, and $A \approx B$ to denote the statement that $A \lesssim B$ and $B \gtrsim A$. the letter $``C"$ will denote a positive constant whose value can change at each appearance. Moreover, the notation ``$\wedge$" will denote the Fourier transform. As usual, for $p\geq 1$, $ \frac{1}{p}+\frac{1}{p'}=1$. $\ell(Q)$ denote the sidelength of $Q$.

\section{Preliminaries}\label{Prelimineries}
\subsection{$A_2$ weights}
We now recall the definition of Muckenhoupt weights.

\begin{defi}  Let $w(x)$ be a nonnegative locally integrable function on $\mathbb{R}^n$. We say $w$ is an $A_2$ weight, written $w \in A_2(\mathbb{R}^{n})$, if
$$
[w]_{A_2}:=\sup _Q \frac{1}{|Q|} \int_Q w(x) d x \cdot \frac{1}{|Q|} \int_Q w(x)^{-1} d x<\infty .
$$
Here the suprema are taken over all cubes $Q \subset \mathbb{R}^n$. The quantity $[w]_{A_2}$ is called the $A_2$ constant of $w$.
\end{defi}
It is well known that $A_2$ weights are doubling. Namely,
\begin{lem}[\cite{G2014}] \label{double}Let $w \in A_2$. Then for every $\lambda>1$ and for every cube $Q \subset \mathbb{R}^n$,
$$
w(\lambda Q) \lesssim \lambda^{2n} w(Q).
$$
\end{lem}
In this article, we will also use the reverse H\"{o}lder inequality for $A_2$ weights.
\begin{lem}[\cite{G2014}]\label{reverse}
Let $w\in A_{2}$. There is a reverse doubling index $\sigma_{w}>0$, such that for every cube $Q\subset \mathbb{R}^n$ yield that
\begin{align}
\bigg[\frac{1}{|Q|}\int_{Q}w^{1+\sigma_{w}}(x)dx  \bigg]^{\frac{1}{1+\sigma_{w}}}\lesssim \frac{w(Q)}{|Q|}.
\end{align}
\end{lem}


\subsection{Dyadic system in $\mathbb{R}^n$}

\begin{defi}
Let the collection $\mathscr{D}^{0}=\mathscr{D}^{0}(\mathbb{R}^n)$ denote the standard system of dyadic cubes on $\mathbb{R}^n$, where
 $$\mathscr{D}^{0}(\mathbb{R}^n)=\bigcup_{k \in \mathbb{Z}} \mathscr{D}^{0}_{k}(\mathbb{R}^n)$$ with
 $$ \mathscr{D}^{0}_{k}(\mathbb{R}^n)=\left\{2^{-k}([0,1)^{n}+m):k\in\mathbb{Z},m\in\mathbb{Z}^{n}\right\} .$$
\end{defi}
Next, we recall a shifted dyadic systems of dyadic cubes on $\mathbb{R}^n$.
\begin{defi}[\cite{HK2012}]  For { $\omega=(\omega_1,\omega_2,\cdots,\omega_n)\in \{0,\frac{1}{3},\frac{2}{3}\}^{n}$}, we can define a bounded number of adjacent dyadic system  $\mathscr{D}^{\omega}=\mathscr{D}^{\omega}(\mathbb{R}^n)$,
 $$
 \mathscr{D}^{\omega}(\mathbb{R}^n)=\bigcup_{k \in \mathbb{Z}} \mathscr{D}^{\omega}_{k}(\mathbb{R}^n),
 $$
 where
 $$  \mathscr{D}^{\omega}_{k}(\mathbb{R}^n)=\left\{2^{-k}\left([0,1)^{n}+m+(-1)^{k}\omega\right):k\in\mathbb{Z},m\in\mathbb{Z}^{n}\right\} .$$
\end{defi}
It is straightforward to check that $\mathscr{D}^{\omega}$ inherits the nestedness property of $\mathscr{D}^{0}$ : if $Q, Q' \in \mathscr{D}^{\omega}$, then $Q \cap Q' \in\{Q, Q', \varnothing\}$(See \cite{HK2012} for more details). When the particular $\omega$ is unimportant, the notation $\mathscr{D}$ is sometimes used for a generic dyadic system.

\subsection{An expression of Haar functions}So let's recall Haar basis on $\mathbb{R}^{n}$.
{  For any dyadic cube $Q\in \mathscr{D}$, there exist dyadic intervals $I_1, I_2, \cdots, I_n$ on $\mathbb{R}$ with common length $l(Q)$, such that $Q=I_1 \times I_2, \cdots\times I_n$. Then $Q$ is associated with $2^n$ Haar functions:
	$$h^{\epsilon}_{Q}(x):=h_{I_1\times I_2\times\cdots\times I_n}^{(\epsilon_1, \epsilon_2,\cdots, \epsilon_n)}(x_1,x_2,\cdots,x_n):=\prod_{i=1}^{n}h^{(\epsilon_i)}_{I_i}(x_i)$$
	where $\epsilon = (\epsilon_1,\epsilon_2,\cdot\cdot\cdot,\epsilon_n)\in \{0,1\}^{n}$ and 
	$$ h^{(1)}_{I_i}:=  \frac{1}{\sqrt{I_i}}\chi_{I_i}  \qquad\text{and}\qquad h^{(0)}_{I_i}:=  \frac{1}{\sqrt{I_i}}(\chi_{I_i-}-\chi_{I_i+})$$
 Writing $\epsilon\equiv 1$ when $\epsilon_{i}\equiv1$ for all $i=1,2,\cdot\cdot\cdot,n$, $h^{1}_{Q}:=\frac{1}{\sqrt{Q}}\chi_Q$ is non-cancellative; on the other hand, when $\epsilon\not\equiv1$,
the rest of the $2^{n}-1$ Haar functions $h^{\epsilon}_{Q}$ associated with $Q$ satisfy the following properties:} 
\begin{lem}\label{haar} For $\epsilon\not\equiv1$, we have

{\rm(1)} $h^{\epsilon}_{Q}$ is supported on $Q$ and $\int_{\mathbb{R}^n}h^{\epsilon}_{Q}(x)dx=0$;

{\rm(2)} $h^{\epsilon}_{Q}$ is constant on each $R\in Ch(Q)$, where $Ch(Q)=\{R\in \mathscr{D}_{k+1}:R\subseteq Q\}$ denotes the dyadic sub-cubes the cube $Q\in \mathscr{D}_k$ ;

{\rm(3)} $\langle h^{\epsilon}_{Q},h^{\eta}_{Q}\rangle=0$, for $\epsilon\not\equiv\eta$;

{\rm(4)} if $h^{\epsilon}_{Q}\neq0$, then$$\|h^{\epsilon}_{Q}\|_{L^{p}(\mathbb{R}^{n})}\approx |Q|^{\frac{1}{p}-\frac{1}{2}}   \quad\text{for}\quad 1\leq p\leq\infty; $$

{\rm(5)} $ \|h^{\epsilon}_{Q}\|_{L^{1}(\mathbb{R}^{n})}\cdot\|h^{\epsilon}_{Q}\|_{L^{\infty}(\mathbb{R}^{n})}= 1;$


{\rm(6)} noting that the average of a function $b$ over a dyadic cube $Q$:
$$ \langle b\rangle_{Q}:=\frac{1}{|Q|}\int_{Q}b(x)dx $$ can be expressed as:
$$ \langle b\rangle_{Q}=\sum_{\substack{P\in \mathscr{D}, Q\subsetneq P\\\epsilon\not\equiv1}}\langle b, h^{\epsilon}_{P}\rangle h^{\epsilon}_{P}(Q). $$
where $h^{\epsilon}_{P}(Q)$ is a constant.

{\rm(7)} fixing a cube $Q$, and expanding $b$ in the Haar basis, we have
$$(b(x)- \langle b\rangle_{Q})\chi_{Q}(x)=\sum_{\substack{R\in \mathscr{D},R\subset Q\\\epsilon\not\equiv1}}\langle b, h^{\epsilon}_{R}\rangle h^{\epsilon}_{R};$$

\end{lem}

\subsection{Characterization of Schatten Class}
In 1989, Rochberg and Semmes \cite{RS1989} introduced the notion
of nearly weakly orthogonal ({\rm NWO}) sequences of functions. 
{ \begin{defi}
	Let $\{e_{Q}\}_{Q\in \mathscr{D}}$ be a collection of functions.  We say $\{e_Q\}_{Q\in \mathscr{D}}$ is {\rm NWO} sequences, if  {\rm supp}$\{e_Q\}\subset Q$ and the maximal function $f^*$ is bounded on $L^{p}(\mathbb{R}^n)$, where $f^*$ is defined as
	$$f^*(x)=\sup_{Q}\frac{|\langle f,e_Q \rangle|}{|Q|^{1/2}}\chi_{Q}(x)$$
\end{defi}

In this paper, we work with weighted versions.
We will use the following result proved by Rochberg and Semmes.}
\begin{lem}[\cite{RS1989}]
If the collection of functions $\left\{e_Q: Q \in \mathscr{D}\right\}$ are supported on $Q$ and satisfy for some $2<r<\infty,\left\|e_Q\right\|_r \lesssim|Q|^{1 / r-1 / 2}$, then $\left\{e_Q\right\}_{Q\in \mathscr{D}}$ is {\rm NWO} sequences.
\end{lem}
If some operator $T$ belongs to $S^{p,q}(\mathcal{G})$, in \cite{RS1989}
Rochberg and Semmes developed a substitute for the Schmidt decomposition of the operator $T$ that will be representations of the form
\begin{align}
T=\sum_{Q\in \mathscr{D}}\lambda_{Q}\langle\cdot,e_{Q}\rangle f_{Q}
\end{align}
with $\{e_{Q}\}_{Q\in \mathscr{D}}$ and $\{f_{Q}\}_{Q\in \mathscr{D}}$ are {\rm NWO} sequences and $\{\lambda_{Q}\}_{Q\in \mathscr{D}}$ is a sequence of scalars. It is easy to see that
\begin{align}\label{eq-NWO1}
\|T\|_{S^{p,q}(\mathcal{G})}\lesssim \|\lambda_Q\|_{\ell^{p,q}}, \quad 0<p<\infty, 0<q<\infty.
\end{align}
When $1<p=q<\infty$, Rochberg and Semmes also obtained
\begin{lem}[\cite{RS1989}]\label{eq-NWO}
{ For any bounded compact operator $T$ on $L^2(\mathbb{R}^{n})$ and $\{e_{Q}\}_{Q\in \mathscr{D}}$ and $\{f_{Q}\}_{Q\in \mathscr{D}}$ are {\rm NWO} sequences, then for $1<p<\infty$, }
$$\left[\sum_{Q\in \mathscr{D}}\left|\langle T e_{Q},f_{Q}\rangle\right|^{p}\right]^{\frac{1}{p}}\lesssim \|T\|_{S^{p}(\mathcal{G})}.$$
\end{lem}
Lacey and the last two authors in  \cite{LLW2022} provided a relationship between Schatten norms on weighted and unweighted $L^{2}(\mathbb{R})$ (which also generalizes to $\mathbb{R}^{n}$).
\begin{lem}[\cite{LLW2022}]\label{Sweight}
Suppose $1\leq p<\infty$, and $w\in A_2$. Then $T$ belongs to $S^{p}(L^{2}(w))$ if and only if $w^{\frac{1}{2}}Tw^{-\frac{1}{2}}$ belong to $S^{p}(L^{2}(\mathbb{R}^n))$.  Moreover,
\begin{align*}
\|T\|_{S^{p}(L^{2}(w))}\approx \|w^{\frac{1}{2}}Tw^{-\frac{1}{2}}\|_{S^{p}(L^{2}(\mathbb{R}^{n}))}.
\end{align*}
\end{lem}
{  Using the idea of Lacey and the last two authors \cite{LLW2022}, i.e., $$M_{w^{\frac{1}{2}}}f(x)=w^{\frac{1}{2}}f(x)$$ is unitary operators and applying the definition of $S^{p,\infty}(L^{2}(w))$, we can also obtian the following result.
\begin{lem}\label{WSweight}
	Suppose $1\leq p<\infty$, and $w\in A_2$. Then $T$ belongs to $S^{p,\infty}(L^{2}(w))$ if and only if $w^{\frac{1}{2}}Tw^{-\frac{1}{2}}$ belong to $S^{p,\infty}(L^{2}(\mathbb{R}^n))$.  Moreover,
	\begin{align*}
		\|T\|_{S^{p,\infty}(L^{2}(w))}\approx \|w^{\frac{1}{2}}Tw^{-\frac{1}{2}}\|_{S^{p,\infty}(L^{2}(\mathbb{R}^{n}))}.
	\end{align*}
\end{lem}

}

\subsection{Description of Besov space.}
\begin{defi}
Suppose $0<p,q<\infty$ and $0<\alpha<1$. Let $b\in L^{1}_{loc}(\mathbb{R}^n)$. Then $b$ belongs to the Besov space $B_{\alpha}^{p,q}(\mathbb{R}^n)$ if
\begin{align*}
\bigg(\int_{\mathbb{R}^n}\bigg(\int_{\mathbb{R}^n}\frac{|b(x)-b(y)|^{p}}{|x-y|^{\frac{p}{q}n+p\alpha}}dy\bigg)^{\frac{q}{p}}dx\bigg)^{\frac{1}{q}}<\infty.
\end{align*}
In particular, note that if $\alpha=\frac{n}{p}$ and $p=q$ then:
$$\|b\|_{B_{n/p}^{p,p}(\mathbb{R}^n)}:=\bigg(\int_{\mathbb{R}^n}\int_{\mathbb{R}^n}\frac{|b(x)-b(y)|^{p}}{|x-y|^{2n}}dydx\bigg)^{\frac{1}{p}}.$$
\end{defi}

Useful in the proof below will be dyadic norms.  We give the norm of the dyadic Besov space next.

\begin{defi}
Suppose $0<p<\infty$. Let $b\in L^{1}_{loc}(\mathbb{R}^n)$ and $\mathscr{D}$ be an arbitrary dyadic system in $\mathbb{R}^n$. Then $b$ belongs to the dyadic Besov space $B_{d}^{p}(\mathbb{R}^n)$ if
\begin{align*}
\|b\|_{B_{d}^{p}(\mathbb{R}^n)}&:=\bigg(\sum_{{\substack{Q\in\mathscr{D}\\\epsilon\not\equiv1}}}\left(|{ \langle b, h^{\epsilon}_{Q}\rangle}||Q|^{-\frac{1}{2}} \right)^p\bigg)^{\frac{1}{p}}<\infty.
\end{align*}
\end{defi}

Suppose $w\in A_2$. By the definition of $A_2$ weights, we obtain that
\begin{align*}
\frac{w(Q)w^{-1}(Q)}{|Q||Q|}\approx1.
\end{align*}
Then
\begin{align*}
\|b\|^{p}_{B_{d}^{p}(\mathbb{R}^n)}&=\sum_{{\substack{Q\in\mathscr{D}\\\epsilon\not\equiv1}}}\left(|{ \langle b, h^{\epsilon}_{Q}\rangle}||Q|^{-\frac{1}{2}} \right)^p\\
&\approx\sum_{{\substack{Q\in\mathscr{D}\\\epsilon\not\equiv1}}}\left(\frac{w(Q)^{\frac{1}{2}}(w^{-1}(Q))^{\frac{1}{2}}|{ \langle b, h^{\epsilon}_{Q}\rangle}|}{|Q|^{\frac{3}{2}}} \right)^p\\
&\approx\sum_{{\substack{Q\in\mathscr{D}\\\epsilon\not\equiv1}}}\left(\frac{|{ \langle b, h^{\epsilon}_{Q}\rangle}||Q|^{\frac{1}{2}}}{w(Q)^{\frac{1}{2}}(w^{-1}(Q))^{\frac{1}{2}}} \right)^p.
\end{align*}

Key to the analysis will be the fact that a suitable family of dyadic norms is equivalent to the norm in the continuous setting, the content of the next lemma.

\begin{lem}\label{lemA1}
Suppose $w\in A_{2}$, $n<p<\infty$. There are dyadic cubes $\mathscr{D}^{0}$ and dyadic cubes { $\mathscr{D}^{\omega}, (\omega\in\{0,\frac{1}{3},\frac{2}{3}\}^n)$} such that
$\bigcap_{\omega\not\equiv0} B_{d}^{p}(\mathbb{R}^{n},\mathscr{D}^{\omega})\bigcap B_{d}^{p}(\mathbb{R}^{n},\mathscr{D}^{0})= B_{n/p}^{p,p}(\mathbb{R}^n)$, that is,
$$\|b\|_{B_{d}^{p}(\mathbb{R}^{n},\mathscr{D}^{0})}
+\sum_{\omega\not\equiv0}\|b\|_{B_{d}^{p}(\mathbb{R}^{n},\mathscr{D}^{\omega})}\approx \|b\|_{B_{n/p}^{p,p}(\mathbb{R}^n)}.$$
\end{lem}
\begin{proof}
On the one hand, we first prove the dyadic Besov norm is dominated by the continuous Besov norm, that is, $$\|b\|_{B_{d}^{p}(\mathbb{R}^{n})}
\lesssim \|b\|_{B_{n/p}^{p,p}(\mathbb{R}^n)}.$$
Choosing a dyadic cube $Q\in \mathscr{D}$. Let $\hat{Q}=Q+\{2\ell(Q)\}^n$. By applying Lemma \ref{haar}, H\"{o}lder's inequality and freely using Lemma \ref{double},
\begin{align*}
|{ \langle b, h^{\epsilon}_{Q}\rangle}||Q|^{-\frac{1}{2}}&\lesssim \bigg|\int_{Q}b(x)h^{\epsilon}_{Q}(x)dx\bigg|\frac{|\hat{Q}|}{|Q|^{\frac{3}{2}}}\\
&\lesssim \int_{\hat{Q}}\int_{Q}|b(x)-b(y)||h^{\epsilon}_{Q}(x)|dxdy|Q|^{-\frac{3}{2}}
\\&\lesssim \bigg(\int_{\hat{Q}}\int_{Q}\frac{|b(x)-b(y)|^{p}}{|x-y|^{2n}}dxdy\bigg)^{\frac{1}{p}}
\bigg(\int_{\hat{Q}}\int_{Q}|h^{\epsilon}_{Q}(x)|^{p'}dxdy\bigg)^{\frac{1}{p'}}{ |Q|^{-\frac{2}{p'}+\frac{1}{2}}}\\
&\lesssim \bigg(\int_{\hat{Q}}\int_{Q}\frac{|b(x)-b(y)|^{p}}{|x-y|^{2n}}dxdy\bigg)^{\frac{1}{p}}.
\end{align*}
Hence, we can obtain
 \begin{align*}
\|b\|^{p}_{B_{d}^{p}(\mathbb{R}^n)}&=\sum_{{\substack{Q\in\mathscr{D}\\\epsilon\not\equiv1}}}\left(|{ \langle b, h^{\epsilon}_{Q}\rangle}||Q|^{-\frac{1}{2}} \right)^p \lesssim \sum_{{Q\in\mathscr{D}}}\int_{\hat{Q}}\int_{Q}\frac{|b(x)-b(y)|^{p}}{|x-y|^{2n}}dxdy 
\\&{ \lesssim \int_{\mathbb R^n}\int_{\mathbb R^n}\frac{|b(x)-b(y)|^{p}}{|x-y|^{2n}}dxdy 
}=  \|b\|^{p}_{B_{n/p}^{p,p}(\mathbb{R}^n)}.
\end{align*}

On the other hand, we need to consider that
\begin{align}\label{eq-S2.4}
\|b\|_{B_{n/p}^{p,p}(\mathbb{R}^n)}\lesssim \|b\|_{B_{d}^{p}(\mathbb{R}^{n},\mathscr{D}^{0})}
+\sum_{\omega\not\equiv0}\|b\|_{B_{d}^{p}(\mathbb{R}^{n},
\mathscr{D}^{\omega})}.
\end{align}
Observe that,
\begin{align*}
\|b\|^{p}_{B_{n/p}^{p,p}(\mathbb{R}^n)}&=\int_{\mathbb{R}^n}\int_{\mathbb{R}^n}
\frac{|b(x)-b(y)|^{p}}{|x-y|^{2n}}dydx\\
&\leq \sum_{k\in\mathbb{ Z}}\sum_{Q\in \mathscr{D}_{k}}\int_{Q}\int_{\{y\in \mathbb{R}^{n}:2^{-k}<|x-y|\leq 2^{-k+1}\}}
\frac{|b(x)-b(y)|^{p}}{|x-y|^{2n}}dydx\\
&\leq \sum_{k\in\mathbb{ Z}}\sum_{Q\in \mathscr{D}_{k}}\int_{Q}\int_{\{y\in \mathbb{R}^{n}:0<|x-y|\leq 3\cdot2^{-k}\}}
\frac{|b(x)-b(y)|^{p}}{|x-y|^{2n}}dydx.
\end{align*}
It is clear that ${\{y\in \mathbb{R}^{n}:0<|x-y|\leq 3\cdot2^{-k}\}}\subset 5Q$. Then there is $J^{\omega}\in\mathscr{D}^{\omega}$ such that $$5Q\subset \bigcup_{ \omega\in\{0,\frac{1}{3},\frac{2}{3}\}^n}J^{\omega}.$$(See \cite{HK2012}). 
Then, we have
\begin{align*}
&\quad\sum_{k\in\mathbb{ Z}}\sum_{Q\in \mathscr{D}_{k}}\int_{Q}\int_{\{y\in \mathbb{R}^{n}:0<|x-y|\leq 3\cdot2^{-k}\}}
\frac{|b(x)-b(y)|^{p}}{|x-y|^{2n}}dydx\\
&\le\sum_{ \omega\in\{0,\frac{1}{3},\frac{2}{3}\}^n}\sum_{J^{\omega}\in\mathscr{D}^{\omega}}
\frac{1}{|J^{\omega}|^2}\int_{J^{\omega}}\int_{J^{\omega}}
|b(x)-b(y)|^{p}dydx\\
&\le\sum_{ \omega\in\{0,\frac{1}{3},\frac{2}{3}\}^n}\sum_{J^{\omega}\in\mathscr{D}^{\omega}}
\frac{1}{|J^{\omega}|^2}\int_{J^{\omega}}\int_{J^{\omega}}
|b(x)-b_{J^{\omega}}|^{p}dxdy
\\&\qquad+\sum_{ \omega\in\{0,\frac{1}{3},\frac{2}{3}\}^n}\sum_{J^{\omega}\in\mathscr{D}^{\omega}}
\frac{1}{|J^{\omega}|^2}\int_{J^{\omega}}\int_{J^{\omega}}
|b(y)-b_{J^{\omega}}|^{p}dydx=:2II.\\
\end{align*}
Using
 $$(b(x)-b_{J^{\omega}})\chi_{J^{\omega}}(x)=\sum_{\substack{P^{\omega}\in \mathscr{D}^{\omega},P^{\omega}\subset J^{\omega}\\\epsilon\not\equiv1}}\langle b, h^{\epsilon}_{P^{\omega}}\rangle h^{\epsilon}_{P^{\omega}}(x)$$
{ and Minkowski's inequality, we have }
\begin{align*}
(II)^{\frac{1}{p}}&\lesssim \bigg(\sum_{ \omega\in\{0,\frac{1}{3},\frac{2}{3}\}^n}\sum_{J^{\omega}\in\mathscr{D}^{\omega}}
\frac{1}{|J^{\omega}|^2}\int_{J^{\omega}}\int_{J^{\omega}}
{ \bigg|}\sum_{\substack{P^{\omega}\in \mathscr{D}^{\omega},P^{\omega}\subset J^{\omega}\\\epsilon\not\equiv1}}\langle b, h^{\epsilon}_{P^{\omega}}\rangle h^{\epsilon}_{P^{\omega}}(x)\bigg|^{p}dxdy\bigg)^{\frac1p}\\
&\lesssim \sum_{ \omega\in\{0,\frac{1}{3},\frac{2}{3}\}^n}\bigg(
\sum_{\substack{P^{\omega}\in \mathscr{D}^{\omega}\\\epsilon\not\equiv1}}\bigg(|\langle b, h^{\epsilon}_{P^{\omega}}\rangle||P^{\omega}|^{-\frac{1}{2}}\bigg)^{p}\bigg)^{\frac1p}.
\end{align*}
Finally we get inequality \eqref{eq-S2.4} and complete the proof of Lemma \ref{lemA1}.
\end{proof}

\section{Proof of (1) in Theorem \ref{Athm1}}\label{Proof1}
From Lemma \ref{lemA1}, we know that the continuous Besov space is the intersection of $3^n$ dyadic Besov spaces with $n<p<\infty$. Thus, the proof of (1) in Theorem \ref{Athm1} can be completed by discussing the following two properties.

\begin{pro}\label{proS1}
For $1<n<p<\infty$, let $w\in A_2$, and $b\in {\rm VMO}(\mathbb{R}^n)$ with $$ \|[b,R_{j}]\|_{S^{p}(L^{2}(w))}<\infty.$$
Then we have
\begin{align*}
\|b\|_{B_{d}^{p}(\mathbb{R}^n)}\lesssim\|[b,R_{j}]\|_{S^{p}(L^{2}(w))}.
\end{align*}
\end{pro}

\begin{pro}\label{proS2}
For $1<n<p<\infty$, suppose that $w\in A_2$, and $b\in B_{d}^{p}(\mathbb{R}^n)$, we have
\begin{align*}
\|[b,R_{j}]\|_{S^{p}(L^{2}(w))}\lesssim\|b\|_{B_{d}^{p}(\mathbb{R}^n)}.
\end{align*}
\end{pro}

\subsection{Proof of Proposition \ref{proS1}}

Below, we consider cubes $Q \in \mathscr{D}$, a fixed dyadic system.
To prove Proposition \ref{proS1}, we first need to use a known result.
\begin{lem}\label{lemA4}
For each dyadic cube $Q$, there exists another dyadic cube $\hat{Q}$ such that

{\rm{(i)}} $|Q|=|\hat{Q}|$, and $\mathrm{distance}(Q,\hat{Q})\approx |Q|$.

{\rm{(ii)}} The kernel of Riesz transform $K_{j}(x-\hat{x})$ does not change sign for all $(x, \hat{x}) \in Q \times \hat{Q}$ and
\begin{align}\label{Kernel}
|K_{j}(x-\hat{x})| \gtrsim \frac{1}{|Q|} .
\end{align}
\end{lem}

Let $m_{b}(\hat{Q})$ be a median value of $b$ over $\hat{Q}$. This means $m_{b}(\hat{Q})$ is a real number such that
\begin{align}\label{median}
E_{1}^{Q}:=\left\{y \in Q: b(y)<m_{b}(\hat{Q})\right\} \quad\text { and }\quad E_{2}^{Q}:=\left\{y \in Q: b(y)>m_{b}(\hat{Q})\right\} .
\end{align}

We note that the upper bound $\left|E_{m}^{Q}\right| \leq \frac{1}{2}|Q|$ for $m=1,2$. A median value always exists, but may not be unique.

By $\int_{Q}h^{\epsilon}_{Q}(x)dx=0$ and using \eqref{median}, a simple calculation gives
$$
\begin{aligned}
\left|\int_{Q} b(x) h^{\epsilon}_{Q}(x) d x\right| &=\left|\int_{Q}\left(b(x)-m_{b}(\hat{Q})\right) h^{\epsilon}_{Q}(x) d x\right|
 \leq \frac{1}{|Q|^{\frac{1}{2}}} \int_{Q}\left|b(x)-m_{b}(\hat{Q})\right| d x \\
& \leq \frac{1}{|Q|^{\frac{1}{2}}} \int_{Q \cap E_{1}^{Q}}\left|b(x)-m_{b}(\hat{Q})\right| d x+\frac{1}{|Q|^{\frac{1}{2}}}  \int_{Q \cap E_{2}^{Q}}\left|b(x)-m_{b}(\hat{Q})\right| d x \\
&=: \operatorname{Term}_{1}^{Q}+\operatorname{Term}_{2}^{Q} .
\end{aligned}
$$
Now we denote
$$
F_{1}^{\hat{Q}}:=\left\{y \in \hat{Q}: b(y) \geq m_{b}(\hat{Q})\right\}\quad \text { and }\quad F_{2}^{\hat{Q}}:=\left\{y \in \hat{Q}: b(y) \leq m_{b}(\hat{Q})\right\} .
$$
Then by the definition of $b_{\hat{Q}}$, we have $\left|F_{1}^{\hat{Q}}\right|=\left|F_{2}^{\hat{Q}}\right| \approx|\hat{Q}|$ and $F_{1}^{\hat{Q}} \cup F_{2}^{\hat{Q}}=\hat{Q}$. Note that for $s=1,2$, if $x \in E_{s}^{Q}$ and $y \in F_{s}^{\hat{Q}}$, then
\begin{align*}
\left|b(x)-m_{b}(\hat{Q})\right| &\leq\left|b(x)-m_{b}(\hat{Q})\right|+\left|m_{b}(\hat{Q})-b(y)\right|\\
&=\left|b(x)-m_{b}(\hat{Q})+m_{b}(\hat{Q})-b(y)\right|=|b(x)-b(y)| .
\end{align*}
Therefore, for $s=1,2$, by using \eqref{Kernel} and by the fact that $|F_{s}^{\hat{Q}}| \approx|Q|$, we have
$$
\begin{aligned}
\operatorname{Term}_{s}^{Q} & \lesssim \frac{1}{|Q|^{\frac{1}{2}}} \int_{Q \cap E_{s}^{Q}}|b(x)-m_{b}(\hat{Q})| d x \frac{|F_{s}^{\hat{Q}}|}{|Q|} \\
&=\frac{1}{|Q|^{\frac{1}{2}}}  \int_{Q \cap E_{s}^{Q}} \int_{F_{s}^{\hat{Q}}}|b(x)-m_{b}(\hat{Q})| \frac{1}{|Q|} d y d x \\
& \lesssim \frac{1}{|Q|^{\frac{1}{2}}} \int_{Q \cap E_{s}^{Q}} \int_{F_{s}^{\hat{Q}}}|b(x)-m_{b}(\hat{Q})||K_j(x-y)| d y d x \\
& \lesssim \frac{1}{|Q|^{\frac{1}{2}}} \int_{Q \cap E_{s}^{Q}} \int_{F_{s}^{\hat{Q}}}|b(x)-b(y)||K_j(x-y)| dy dx.
\end{aligned}
$$

To continue, by noting that $K_j(x-y)$ and $b(x)-b(y)$ does not change sign for $(x, y) \in\left(Q \cap E_{s}^{Q}\right) \times F_{s}^{\hat{Q}}$, $s=1,2$, we have that
$$
\begin{aligned}
\operatorname{Term}_{s}^{Q} & \lesssim \frac{1}{|Q|^{\frac{1}{2}}} \left|\int_{Q \cap E_{s}^{Q}} \int_{F_{s}^{Q}}(b(x)-b(y)) K_j(x-y) d y d x\right| \\
&=\frac{1}{|Q|^{\frac{1}{2}}} \left|\int_{\mathbb{R}^n} \int_{\mathbb{R}^n}(b(x)-b(y)) K_j(x-y) \chi_{F_{s}^{\hat{Q}}}(y) d y \chi_{Q \cap E_{s}^{Q}}(x) d x\right| .
\end{aligned}
$$
We now insert the weight $w$ to get
\begin{align*}
\operatorname{Term}_{s}^{Q}
&\lesssim \frac{1}{|Q|^{\frac{1}{2}}} \bigg|\int_{\mathbb{R}^n} \int_{\mathbb{R}^n}(b(x)-b(y)) w^{\frac{1}{2}}(x) K_j(x-y) w^{-\frac{1}{2}}(y)
\bigg(w^{\frac{1}{2}}(y) \chi_{F_{s}^{\hat{Q}}}(y)\bigg) d y
\bigg(w^{-\frac{1}{2}}(x) \chi_{Q \cap E_{s}^{Q}}(x)\bigg) d x\bigg| .
\end{align*}
Thus, we further have
$$
\begin{aligned}
&{\sum_{Q \in \mathscr{D},\epsilon \not \equiv 1}} \left(\frac{|\langle b, h^{\epsilon}_{Q} \rangle||Q|^{\frac{1}{2}}}{w(Q)^{\frac{1}{2}}(w^{-1}(Q))^{\frac{1}{2}}} \right)^p \\
&\lesssim{ \sum_{Q \in \mathscr{D}, \epsilon\not\equiv1} \sum_{s=1}^{2}}\Bigg|\int_{\mathbb{R}^{n}} \int_{\mathbb{R}^{n}}(b(x)-b(y)) w^{\frac{1}{2}}(x) K_j(x-y) w^{-\frac{1}{2}}(y)\bigg(w^{\frac{1}{2}}(y) \chi_{F_{s}^{\hat{Q}}}(y)\bigg) d y\frac{\Big(w^{-\frac{1}{2}}(x) \chi_{Q \cap E_{s}^{Q}}(x)\Big)}{w(Q)^{\frac{1}{2}}(w^{-1}(Q))^{\frac{1}{2}}} d x\Bigg|^{p} \\
&\lesssim {\sum_{Q \in \mathscr{D}, \epsilon\not\equiv1} \sum_{s=1}^{2}}\bigg|\int_{\mathbb{R}^{n}} \int_{\mathbb{R}^{n}}(b(x)-b(y)) w^{\frac{1}{2}}(x) K_j(x-y) w^{-\frac{1}{2}}(y) \frac{w^{\frac{1}{2}}(y) \chi_{F_{s}^{\hat{Q}}}(y)}{w(\hat{Q})^{\frac{1}{2}}} d y \frac{w^{-\frac{1}{2}}(x) \chi_{Q \cap E_{s}^{Q}}(x)}{(w^{-1}(Q))^{\frac{1}{2}}} d x\bigg|^{p}
\\&=:\sum_{Q \in \mathscr{D}, \epsilon\not\equiv1} \sum_{s=1}^{2} \left|\left\langle w^{\frac{1}{2}}[b, R_{j}] w^{-\frac{1}{2}} G^{s}_{\hat{Q}}, H^{s}_{Q}\right\rangle\right|^{p},
\end{aligned}
$$
where
$$
G^{s}_{\hat{Q}}(y):=\frac{w^{\frac{1}{2}}(y) \chi_{F_{s}^{\hat{Q}}}(x)}{w(\hat{Q})^{\frac{1}{2}}} \quad \text { and } \quad H^{s}_{Q}(x):= \frac{w^{-\frac{1}{2}}(x) \chi_{Q \cap E_{s}^{Q}}(x)}{(w^{-1}(Q))^{\frac{1}{2}}}.
$$
Applying Lemma \ref{reverse}, there is a reverse doubling index $\sigma_{w}>0$, we have
\begin{align*}
\|G^{s}_{\hat{Q}}\|_{L^{2(\sigma_{w}+1)}}&\lesssim\frac{1}{w(\hat{Q})^{\frac{1}{2}}}\bigg(\int_{\hat{Q}}w^{(\sigma_{w}+1)}(x) dx\bigg)^{\frac{1}{2(\sigma_{w}+1)}}\lesssim |\hat{Q}|^{\frac{1}{2(\sigma_{w}+1)}-\frac{1}{2}}.
\end{align*}
Similarly, $\|H^{s}_{Q}\|_{L^{2(\sigma_{w}+1)}}\lesssim|Q|^{\frac{1}{2(\sigma_{w}+1)}-\frac{1}{2}}$. Then, $\{G^{s}_{\hat{Q}}\}_{\hat{Q}\in\mathscr{D}}$ and $\{H^{s}_{Q}\}_{Q\in\mathscr{D}}$ are NWO sequences for $L^{2}(\mathbb{R}^n)$.
It follows from Lemma \ref{Sweight} and Lemma \ref{eq-NWO} that
$$
\|b\|_{B_{d}^{p}(\mathbb{R}^n)}\lesssim\|w^{\frac{1}{2}}[b, R_{j}] w^{-\frac{1}{2}}\|_{S^{p}\left(L^{2}(\mathbb{R}^n)\right)}\approx\|[b, R_{j}] \|_{S^{p}\left(L^{2}(w)\right)} .
$$
The proof of Proposition \ref{proS1} is complete.

\subsection{Proof of Proposition \ref{proS2}}
In \cite{CDW2019}, the authors have obtained that: for $w\in A_2$, for every $b\in B_{n/p}^{p,p}(\mathbb{R}^n)\subset {\rm{VMO}}$, $[b,R_{j}]$ is compact form $L^{2}(w)$ to $L^{2}(w)$. On the other hand, Petermichl, Treil and Volberg have given that Riesz transforms are averages of the dyadic shift in \cite{PTV2002}. See also \cite{Pe2008}.
On a fixed dyadic cubes $\mathscr{D}$ with Haar basis $\{h^{\varepsilon}_{Q}\}$, {  let $\sigma:\mathscr{D}\rightarrow \mathscr{D}$ such $|\sigma(Q)|=2^{-n}|\sigma(Q)|$, for all $Q\in \mathscr{D}$. Using the same notation for a map $
\sigma: \{0,1\}^{n}-\{1\}^n \longrightarrow \{\{0,1\}^{n}-\{1\}^n\} \cup\{0\},$ if $\sigma(\epsilon)=0$ then $h^{\sigma(\epsilon)}:=0$. In \cite{LPPW2008}, the authors define a dyadic shift operator $\mathrm{III}$ by
\begin{align}
	\mathrm{III}f(x):=\sum_{Q\in\mathscr{D},\epsilon\not\equiv1}\langle f,h^{\epsilon}_{Q} \rangle h^{\sigma(\epsilon)}_{\sigma(Q)}(x).
\end{align}}
It is clear that
$$\mathrm{III}h^{\epsilon}_Q=h^{\sigma(\epsilon)}_{\sigma(Q)}.$$
 We further have $\|\mathrm{III}\|_{L^{2}(w)\rightarrow L^{2}(w)}\lesssim 1$ and the Riesz transforms are in the convex hull of the class of operators $\mathrm{III}$.
Therefore, we only need to prove that
$$
\|[b, \mathrm{III}]\|_{S^{p}\left(L^{2}(w)\right)} \lesssim\|b\|_{B_{d}^{p}(\mathbb{R}^n)}.
$$
As proved by \cite{HLW2017}, $[b, \mathrm{III}]$ is decomposed into ``paraproducts'' operators as follows
\begin{align*}
(\Pi^{\mathscr{D}}_{b}+\Pi^{*\mathscr{D}}_{b}+\Gamma^{\mathscr{D}}_{b})(\mathrm{III}f)
{ -}\mathrm{III}(\Pi^{\mathscr{D}}_{b}+\Pi^{*\mathscr{D}}_{b}+\Gamma^{\mathscr{D}}_{b})f+\Pi^{\mathscr{D}}_{\mathrm{III}f}b-\mathrm{III}(\Pi^{\mathscr{D}}_{f}b),
\end{align*}
where
\begin{align*}
\Pi^{\mathscr{D}}_{b}f=\sum_{Q\in\mathscr{D},\epsilon\not\equiv 1}\langle b, h^{\epsilon}_{Q}\rangle{ \langle f\rangle_{Q}}h^{\epsilon}_{Q},
\end{align*}
\begin{align*}
\Pi^{*\mathscr{D}}_{b}f=\sum_{Q\in\mathscr{D},\epsilon\not\equiv 1}\langle b, h^{\epsilon}_{Q}\rangle\langle f, h^{\epsilon}_{Q}\rangle\frac{\chi_{Q}}{|Q|}
\end{align*}
and
\begin{align*}
\Gamma^{\mathscr{D}}_{b}f=\sum_{Q\in\mathscr{D}}\sum_{\substack{\epsilon,\eta\not\equiv 1\\\epsilon\not\equiv\eta}}\langle b, h^{\epsilon}_{Q}\rangle\langle f, h^{{ \eta}}_{Q}\rangle h^{\epsilon}_{Q}h^{\eta}_{Q}
\end{align*}
are the ``paraproduct" operators with symbol $b$. Then we have
\begin{align*}
\|[b, \mathrm{III}]\|_{S^{p}\left(L^{2}(w)\right)}
&\leq2\|\Pi^{\mathscr{D}}_{b}\|_{S^{p}\left(L^{2}(w)\right)}\|\mathrm{III}\|_{L^{2}(w)\rightarrow L^{2}(w)}
\\&\quad+2\|\Pi^{*\mathscr{D}}_{b}\|_{S^{p}\left(L^{2}(w)\right)}\|\mathrm{III}\|_{L^{2}(w)\rightarrow L^{2}(w)}
\\&\quad+2\|\Gamma^{\mathscr{D}}_{b}\|_{S^{p}\left(L^{2}(w)\right)}\|\mathrm{III}\|_{L^{2}(w)\rightarrow L^{2}(w)}
\\&\quad+\|\Pi^{\mathscr{D}}_{\mathrm{III}f}b-\mathrm{III}(\Pi^{\mathscr{D}}_{f}b)\|_{S^{p}\left(L^{2}(w)\right)}.
\end{align*}
Thus, in order to show that Proposition \ref{proS2}, Using Lemma \ref{Sweight}, we need to obtain the following two lemmas.
\begin{lem}\label{LemS4}
For $1<n<p<\infty$, suppose that $w\in A_2$, and $b\in {\rm{VMO}(\mathbb{R}^n)}$, we have{  $w^{\frac{1}{2}}\Pi^{\mathscr{D}}_{b}w^{-\frac{1}{2}}$,$w^{\frac{1}{2}}\Pi^{*\mathscr{D}}_{b}w^{-\frac{1}{2}}$, and $w^{\frac{1}{2}}\Gamma^{\mathscr{D}}_{b}w^{-\frac{1}{2}}$ belong to $S^{p}(L^{2}(\mathbb{R}^{n}))$ respectively, if and only if $b\in B_{d}^{p}(\mathbb{R}^n)$, that is,}
\begin{align}\label{eqLem4-1}
\|w^{\frac{1}{2}}\Pi^{\mathscr{D}}_{b}w^{-\frac{1}{2}}\|_{S^{p}(L^{2}(\mathbb{R}^{n}))}{ \approx}\|b\|_{B_{d}^{p}(\mathbb{R}^n)};
\end{align}
\begin{align}\label{eqLem4-2}
\|w^{\frac{1}{2}}\Pi^{*\mathscr{D}}_{b}w^{-\frac{1}{2}}\|_{S^{p}(L^{2}(\mathbb{R}^{n}))}{ \approx}\|b\|_{B_{d}^{p}(\mathbb{R}^n)};
\end{align}
\begin{align}\label{eqLem4-3}
\|w^{\frac{1}{2}}\Gamma^{\mathscr{D}}_{b}w^{-\frac{1}{2}}\|_{S^{p}(L^{2}(\mathbb{R}^{n}))}{ \approx}\|b\|_{B_{d}^{p}(\mathbb{R}^n)}.
\end{align}
\end{lem}

\begin{proof}\eqref{eqLem4-1} and \eqref{eqLem4-2} in Lemma \ref{LemS4} has been proved by Lacey and the last two authors in \cite{LLW2022} in one dimension, which also generalizes to $n$ dimensions. It is therefore only necessary to prove \eqref{eqLem4-3}.

\textbf{Sufficiency:} Suppose $b\in B^{p}_{d}(\mathbb{R}^n)$. By the definition of $\Gamma^{\mathscr{D}}_{b}f$, we have
\begin{align*}
&~\left(w^{\frac{1}{2}}\Gamma^{\mathscr{D}}_{b}w^{-\frac{1}{2}}\right)(f)(x)\\
&=\sum_{Q\in\mathscr{D}}\sum_{\substack{\epsilon,\eta\not\equiv 1\\\epsilon\not\equiv\eta}}\langle b, h^{\epsilon}_{Q}\rangle\langle w^{-\frac{1}{2}}f, h^{{ \eta}}_{Q}\rangle h^{\epsilon}_{Q}(x)h^{\eta}_{Q}(x)w^{\frac{1}{2}}(x)\\
&=C_{[w]_{A_2}} \sum_{Q\in\mathscr{D}}\sum_{\substack{\epsilon,\eta\not\equiv 1\\\epsilon\not\equiv\eta}}\frac{\langle b, h^{\epsilon}_{Q}\rangle w(Q)^{\frac{1}{2}}(w^{-1}(Q))^{\frac{1}{2}}}{|Q|^{\frac{3}{2}}}\cdot\frac{h^{\epsilon}_{Q}(x)h^{\eta}_{Q}(x)w^{\frac{1}{2}}(x)|Q|}{w(Q)^{\frac{1}{2}}}
\int_{\mathbb{R}^n} \frac{w^{-\frac{1}{2}}(y)h^{{ \eta}}_{Q}(y)|Q|^{\frac{1}{2}}}{(w^{-1}(Q))^{\frac{1}{2}}}f(y)dy\\
&=:C_{[w]_{A_2}} \sum_{Q\in\mathscr{D}}\sum_{\substack{\epsilon,\eta\not\equiv 1\\\epsilon\not\equiv\eta}}B(Q)\cdot { G_Q(x)}\int_{\mathbb{R}^n}f(y)H_{Q}(y)dy.
\end{align*}
{ As proved in the proof of Proposition \ref{proS1}}, applying Lemma \ref{reverse} leads to ${ \{G_Q\}_{Q\in\mathscr{D}}}$ and ${ \{H_Q\}_{Q\in\mathscr{D}}}$ are NWO sequences for $L^2(\mathbb{R}^n)$, and by Lemma \ref{Sweight} and \eqref{eq-NWO1} imply that
\begin{align*}
\|\Gamma^{\mathscr{D}}_{b}\|^{p}_{S^{p}(L^{2}(w))}
=\|w^{\frac{1}{2}}\Gamma^{\mathscr{D}}_{b}w^{-\frac{1}{2}}\|^{p}_{S^{p}(L^{2}(\mathbb{R}^{n}))}
\leq \sum_{\substack{Q\in\mathscr{D}\\\epsilon\not\equiv1}}|B_{Q}|^{p}\approx\|b\|^{p}_{B_{n/p,d}^{p,p}(\mathbb{R}^n)}.
\end{align*}

\textbf{Necessity:}. For any dyadic cubes $Q$, we have { $\langle b,h^{\epsilon}_Q\rangle=\langle\Gamma^{\mathscr{D}}_{b}(h^{\eta}_{Q}),h^{\epsilon}_{Q}h^{\eta}_{Q}|Q|\rangle$}, Therefore,
\begin{align*}
 &~\sum_{{\substack{Q\in\mathscr{D}\\\epsilon\not\equiv1}}}\left(\frac{|\langle\Gamma^{\mathscr{D}}_{b}(h^{\eta}_{Q}),{ h^{\epsilon}_{Q}h^{\eta}_{Q}|Q|}\rangle||Q|^{\frac{1}{2}}}{w(Q)^{\frac{1}{2}}(w^{-1}(Q))^{\frac{1}{2}}} \right)^p\\
&=\sum_{{\substack{Q\in\mathscr{D}\\\epsilon\not\equiv1}}}\left(\frac{|
\langle w^{\frac{1}{2}}\Gamma^{\mathscr{D}}_{b}w^{-\frac{1}{2}}(w^{\frac{1}{2}}h^{\eta}_{Q}),{ w^{-\frac{1}{2}}h^{\epsilon}_{Q}h^{\eta}_{Q}|Q|}\rangle||Q|^{\frac{1}{2}}}{w(Q)^{\frac{1}{2}}(w^{-1}(Q))^{\frac{1}{2}}} \right)^p\\
&=\sum_{{\substack{Q\in\mathscr{D}\\\epsilon\not\equiv1}}}\left|\left\langle w^{\frac{1}{2}}\Gamma^{\mathscr{D}}_{b}w^{-\frac{1}{2}}\left(\frac{w^{\frac{1}{2}}\sqrt{|Q|}h^{\epsilon}_{Q}}{w(Q)^{\frac{1}{2}}}\right),
\frac{{ w^{-\frac{1}{2}}h^{\epsilon}_{Q}h^{\eta}_{Q}|Q|}}{(w^{-1}(Q))^{\frac{1}{2}}}\right\rangle \right|^p
\\&=:\sum_{{\substack{Q\in\mathscr{D}\\\epsilon\not\equiv1}}}\left|\left\langle w^{\frac{1}{2}}\Gamma^{\mathscr{D}}_{b}w^{-\frac{1}{2}}\left(G'_{Q}\right),
H'_{Q}\right\rangle \right|^p,
\end{align*}
where $$G'_{Q}:=\frac{w^{\frac{1}{2}}\sqrt{|Q|}h^{\epsilon}_{Q}}{w(Q)^{\frac{1}{2}}}
\quad \text{and}\quad
H'_{Q}:=\frac{w^{-\frac{1}{2}}{ h^{\epsilon}_{Q}h^{\eta}_{Q}}|Q|}{(w^{-1}(Q))^{\frac{1}{2}}}.$$
By Lemma \ref{reverse}, similar to the proof of Proposition \ref{proS1}, we can easily obtain that the above two collections of functions are NWO sequences. Thus, we establish by Lemma \ref{eq-NWO} that
\begin{align*}
\|b\|^{p}_{B^{p}_{d}(\mathbb{R}^n)}
&\approx\sum_{{\substack{Q\in\mathscr{D}\\\epsilon\not \equiv1}}}\left(\frac{|\langle b, h^{\epsilon}_{Q}\rangle ||Q|^{\frac{1}{2}}}{w(Q)^{\frac{1}{2}}(w^{-1}(Q))^{\frac{1}{2}}} \right)^p.
\end{align*}
The proof is complete.
\end{proof}

\begin{lem}\label{LemS5}
For $1<n<p<\infty$, suppose that $w\in A_2$, and $b\in B_{d}^{p}(\mathbb{R}^n)$. Let $\mathfrak{R}f:=\Pi_{\mathrm{III}f}b-\mathrm{III}(\Pi_{f}b)$, we have
\begin{align}\label{eqLemS5}
\|w^{\frac{1}{2}}\mathfrak{R}w^{-\frac{1}{2}}\|_{S^{p}(L^{2}(\mathbb{R}^{n}))}\lesssim\|b\|_{B_{d}^{p}(\mathbb{R}^n)}.
\end{align}
\end{lem}
\begin{proof}
Computation gives
\begin{align*}
\mathfrak{R}f(x)&:=\Pi^{\mathscr{D}}_{\mathrm{III}f}b-\mathrm{III}(\Pi^{\mathscr{D}}_{f}b)\\
&=\sum_{P\in\mathscr{D},\eta\not\equiv 1}\langle \mathrm{III}f, h^{\eta}_{P}\rangle  { \langle b\rangle_{P}}h^{\eta}_{P}(x)-\sum_{Q\in\mathscr{D},\epsilon\not\equiv 1}\langle \Pi^{\mathscr{D}}_{f}b, h^{\epsilon}_{Q}\rangle h^{\sigma(\epsilon)}_{{ \sigma}(Q)}(x)
\\&=\sum_{Q\in\mathscr{D},\epsilon\not\equiv 1}\langle f, h^{\epsilon}_{Q}\rangle
{ (\langle b\rangle_{\sigma(Q)}-\langle b\rangle_{Q})}h^{\sigma(\epsilon)}_{\sigma(Q)}(x)
\\&=\sum_{Q\in\mathscr{D},\epsilon\not\equiv 1}\langle f, h^{\epsilon}_{Q}\rangle
{ \langle b,h^{\epsilon}_{Q}\rangle h^{\epsilon}_{Q}(\sigma(Q))}h^{\sigma(\epsilon)}_{\sigma(Q)}(x),
\end{align*} { where the last inequality form \cite[(2.2)]{HLW2017}}. Therefore, 
we have that
\begin{align*}
&~w^{\frac{1}{2}}(x)\mathfrak{R}(w^{-\frac{1}{2}}f)(x)
\\&=\sum_{Q\in\mathscr{D},\epsilon\not\equiv 1}\int_{\mathbb{R}^n} w^{-\frac{1}{2}}(y)f(y)h^{\epsilon}_{Q}(y)dy w^{\frac{1}{2}}(x){ \langle b,h^{\epsilon}_{Q}\rangle h^{\epsilon}_{Q}(\sigma(Q))}h^{\sigma(\epsilon)}_{{ \sigma}(Q)}(x)
\\&=\sum_{Q\in\mathscr{D},\epsilon\not\equiv 1}\int_{\mathbb{R}^n} w^{-\frac{1}{2}}(y)f(y)h^{\epsilon}_{Q}(y)dy w^{\frac{1}{2}}(x){ \langle b,h^{\epsilon}_{Q}\rangle h^{\epsilon}_{Q}(\sigma(Q))} h^{\sigma(\epsilon)}_{{ \sigma}(Q)}(x)
\\&= C_{[w]_{A_2}}\sum_{Q\in\mathscr{D},\epsilon\not\equiv 1}\frac{{ \langle b,h^{\epsilon}_{Q}\rangle h^{\epsilon}_{Q}(\sigma(Q))} w(Q)^{\frac{1}{2}}(w^{-1}(Q))^{\frac{1}{2}}}{|Q|}\cdot\frac{h^{\sigma(\epsilon)}_{{ \sigma}(Q)}(x)w^{\frac{1}{2}}(x)|Q|^{\frac{1}{2}}}{w(Q)^{\frac{1}{2}}}
\\&\qquad\quad\times\int_{\mathbb{R}^n} \frac{w^{-\frac{1}{2}}(y)h^{\epsilon}_{Q}(y)|Q|^{\frac{1}{2}}}{(w^{-1}(Q))^{\frac{1}{2}}}f(y)dy\\
&=: C_{[w]_{A_2}}\sum_{Q\in\mathscr{D},\epsilon\not\equiv 1}B^{1}(Q)\cdot G^{1}_{Q}(x)\int_{\mathbb{R}^n}f(y)H^{1}_{Q}(y)dy,
\end{align*}
where $$B^{1}(Q):=\frac{{ \langle b,h^{\epsilon}_{Q}\rangle h^{\epsilon}_{Q}(\sigma(Q))} w(Q)^{\frac{1}{2}}(w^{-1}(Q))^{\frac{1}{2}}}{|Q|},$$$$ G^{1}_{Q}(x):=\frac{h^{\sigma(\epsilon)}_{{ \sigma}(Q)}(x)w^{\frac{1}{2}}(x)|Q|^{\frac{1}{2}}}{w(Q)^{\frac{1}{2}}}
\quad\text{and}\quad H^{1}_{Q}(y):=\frac{w^{-\frac{1}{2}}(y)h^{\epsilon}_{Q}(y)|Q|^{\frac{1}{2}}}{(w^{-1}(Q))^{\frac{1}{2}}}.$$
By Lemma \ref{reverse}, similar to the proof of Proposition \ref{proS1}, we know that $\{G^{1}_{R}\}_{R\in \mathscr{D}}$ and $\{H^{1}_{Q}\}_{Q\in \mathscr{D}}$ are NWO sequences for $L^2(\mathbb{R}^n)$.
Therefore, by \eqref{eq-NWO1}, and $ { |h^{\epsilon}_{Q}(\sigma(Q))|\approx |Q|^{-\frac{1}{2}}}$ we get
\begin{align*}
\|w^{\frac{1}{2}}\mathfrak{R}w^{-\frac{1}{2}}\|_{S^{p}(L^{2}(\mathbb{R}^{n}))}\lesssim\|B^{1}(Q)\|_{\ell^p}
&{ \approx}\bigg(\sum_{{\substack{Q\in\mathscr{D}\\\epsilon\not\equiv1}}}\left(\frac{|{ \langle b, h^{\epsilon}_{Q}\rangle} ||Q|^{\frac{1}{2}}}{w(Q)^{\frac{1}{2}}(w^{-1}(Q))^{\frac{1}{2}}} \right)^p\bigg)^{\frac{1}{p}}\approx\|b\|_{B_{d}^{p}(\mathbb{R}^n)}.
\end{align*}
The proof is complete.
\end{proof}

\section{Proof of Theorem \ref{Athm1}: $0<p\leq n$}\label{Proof2}

In this section, we prove {\rm(2)} in Theorem \ref{Athm1}. That is, for $0 < p \le n$, the commutator $[b, R_{j}] \in S^{p}(L^{2}(w))$ if and only if $b$ is a constant. The sufficient
condition is obvious, since $[b, R_{j}] =0$ when $b$ is a constant. Thus, it suffices to show the necessary
condition. { By the inclusion $S^{p}(L^{2}(w)) \subset S^{q}(L^{2}(w))$ for $p < q$, then the proof of (2) in Theorem \ref{Athm1} can be proved on the basis of the following property.}

\begin{pro}\label{proS3}
Suppose { $n>1$}, $w\in A_2$, and $b\in {\rm{VMO}}(\mathbb{R}^n)$ with
$[b, R_{j}] \in S^{n}(L^{2}(w))$, then $b$ is a constant.
\end{pro}
\begin{proof}
Similar to the proof in Section 3.1, { applying Lemma \ref{eq-NWO}, for $n>1$,} we know that
\begin{align}\label{eq-S8}
\bigg(\sum_{Q \in \mathscr{D}}\left(R_{Q} \right)^n\bigg)^{\frac{1}{n}} \lesssim  \|[b,R_{j}]\|_{S^{n}(L^{2}(w))},
\end{align}
where
\begin{align*}
R_{Q}:= \frac{1}{|Q|}\int_{Q}|b(x)-\langle b\rangle_{\hat{Q}}|dx.
\end{align*}
Here , $Q$ and $\hat{Q}$ are the dyadic cubes chosen in Section 3.1. If $b$ is not constant and $b\in C^{\infty}(\mathbb{R}^{n})$, it is clear
 that there exists a point $x_0\in \mathbb{R}^{n}$ such that $\nabla b(x_0)\neq0$. By applying \cite[Lemma 5.3]{FLL2021} with $\mathbb{R}^{n}$, then there is $\varepsilon>0$ and $N>0$ such that if $k>N$, then for any dyadic cube $\tilde{Q}\in \mathscr{D}_{k}$ with { $Q\subset \tilde{Q}$, $\hat{Q}\subset \tilde{Q}$}, and satisfying $|C(\tilde{Q})-x_{0}|<\varepsilon$,
\begin{align*}
R_{Q}\geq |\langle b\rangle_{Q}- \langle b\rangle_{\hat{Q}}|\geq C\ell(\tilde{Q})|\nabla b(x_0)|.
\end{align*}
Noting that for $k>N$, we obtain
\begin{align*}
\bigg(\sum_{Q \in \mathscr{D}}\left(R_{Q} \right)^n\bigg)^{\frac{1}{n}} =\infty.
\end{align*}
It is in contradiction with \eqref{eq-S8}. Therefore, the proposition holds.
\end{proof}

\section{Proof of Theorem \ref{Athm2}:  $p =n$}\label{Proof3}
\subsection{Proof of the sufficient condition}
In this subsection, we assume that $b\in \dot{W}^{1,n}(\mathbb{R}^{n})$, then prove that $$[b, R_{j}]\in S^{p,\infty}\left(L^{2}(w)\right).$$ By { Lemma \ref{WSweight}}, we just need to show that $\|w^{\frac{1}{2}}[b, R_{j}] (w^{-\frac{1}{2}})\|_{ S^{n,\infty}\left(L^{2}(\mathbb{R}^n)\right)}\lesssim \|b\|_{\dot{W}^{1,n}(\mathbb{R}^{n})}$.

Let $\Lambda=\{(x,y)\in \mathbb{R}^{n}\times \mathbb{R}^{n}:x=y\}$, and $\Omega=\{(x,y)\in\mathbb{R}^{n}\times \mathbb{R}^{n}\backslash \Lambda :x\neq y\}$. Assuming that $\mathscr{P}$ is a dyadic Whitney decomposition family of the open set $\Omega$, that is $\bigcup_{P\in\mathscr{P}}=\Omega$. Therefore, we write $K_{j}(x-y)=\Sigma_{P\in\mathscr{P}}K_{j}(x-y)\chi_P(x,y)$, and $P$ can be the cubes $P_{1}\times P_{2}$, where $P_{1},P_{2}\in \mathscr{D}$, have the same side length and that distance between them must be comparable to this sidelength. Thus, for each dyadic cube $P_1\in \mathscr{D}$, $P_{2}$ is related to $P_1$ and at most $M$ of the cubes $P_{2}$ such that $P_{1}\times P_{2}\in \mathscr{P}$.
Therefore, let $Q=P_{1}$, there is $R_{Q, s}$ such that $Q\times R_{Q, s}\in\mathscr{P}$, where $s=1,2,\cdot\cdot\cdot,M$, we can reorganize the sum
\begin{align*}
K_{j}(x-y)=\sum_{P\in\mathscr{P}}K_{j}(x-y)\chi_P(x,y)
=\sum_{Q\in\mathscr{D}}\sum_{s=1}^{M}K_{j}(x-y)\chi_{(Q\times R_{Q, s})}(x,y),
\end{align*}
where $|Q|=|R_{Q, s}|$ and $\mathrm{distance}(Q,R_{Q, s})\approx |Q|$.

Next, using the multiple Fourier series on $Q\times R_{Q, s}$ we can write,
\begin{eqnarray*}
K_{j}(x-y)\chi_{(Q\times R_{Q, s})}(x,y)
&=&\sum_{\vec{l}\in\mathbb{Z}^{2n}}c^{j}_{\vec{l}}e^{2\pi i\vec{l}'\cdot \widetilde{x}}e^{2\pi i\vec{l}''\cdot \widetilde{y}}\chi
_{Q}(x)\cdot \chi _{Q_{R,s}}(y),
\end{eqnarray*}
where $x_{i}=C_{Q}^{(i)}+\ell(Q)\tilde{x}_{i}$, $y_{i}=C_{R_{Q,s}}^{(i)}+\ell(R_{Q,s})\tilde{y}_{i}$, $i=1,2,\cdot\cdot\cdot,n$, and $\vec{l}=(\vec{l}', \vec{l}'')$ where $\vec{l}'=(l_1,l_2,\cdot\cdot\cdot,l_n)$, $\vec{l}''=(l_{n+1},l_{n+2},\cdot\cdot\cdot,l_{2n})$, and
\[
c^j_{\vec{l}}=\int_{R_{Q,s}}\int_{Q}K_{j}(x-y)\chi_{(Q\times R_{Q, s})}(x,y)e^{-2\pi i\vec{l}'\cdot \widetilde{x}
}e^{-2\pi i\vec{l}''\cdot \widetilde{y}}dxdy\frac{1}{|Q|}\frac{1}{|R_{Q,s}|}.
\]
For the multi-index $\alpha,~\beta\in \mathbb{Z}_+^{n}$, we apply $\hat{f}(\vec{l})=\frac{1}{(2\pi i\vec{l})^{\alpha}}\widehat{(\partial^{\alpha}f)}(\vec{l})$, and the size condition of $K_{j}(x-y)$, $$|\partial^{\alpha}_{x}\partial^{\beta}_{y}K_{j}(x-y)|\leq C(\alpha,\beta)\frac{1}{|x-y|^{n+|\alpha|+|\beta|}},$$
yield that
 \begin{align*}
 |c^j_{\vec{l}}|&\lesssim \frac{1}{(1+|\vec{l}|)^{|\alpha|+|\beta|}}\ell(Q)^{|\alpha|}\ell(R_{Q,s})^{|\beta|}\int_{R_{Q,s}}\int_{Q} |\partial^{\alpha}_{x}\partial^{\beta}_{y}K_{j}(x-y)|dxdy\frac{1}{|Q|}\frac{1}{|R_{Q,s}|}\\
 &\lesssim \frac{1}{(1+|\vec{l}|)^{|\alpha|+|\beta|}}\ell(Q)^{|\alpha|}\ell(R_{Q,s})^{|\beta|}\int_{R_{Q,s}}\int_{Q} \frac{1}{|x-y|^{n+|\alpha|+|\beta|}}dxdy\frac{1}{|Q|}\frac{1}{|R_{Q,s}|}\\
&\lesssim\frac{1}{|Q|}\frac{1}{(1+|\vec{l}|)^{|\alpha|+|\beta|}}.
\end{align*}
Let $\lambda^j_{\vec{l},Q}=|Q|^{\frac{1}{2}}|R_{Q,s}|^{\frac{1}{2}}c^j_{\vec{l}}$, then $$|\lambda^j_{\vec{l},Q}|\lesssim \frac{1}{(1+|\vec{l}|)^{|\alpha|+|\beta|}},$$ and
\begin{eqnarray*}
K_j(x-y)\chi _{(Q\times R_{Q, s})}(x,y)=\sum_{\vec{l}\in\mathbb{Z}^{2n}}\lambda^j_{\vec{l},Q}\frac{1}{|Q|^{1/2}}F_{\vec{l}',Q}(x)\frac{1}{|R_{Q,s}|^{1/2}}G_{\vec{l}'',R_{Q,s}}(y),
\end{eqnarray*}
where
$F_{\vec{l}',Q}(x)=e^{2\pi i\vec{l}'\cdot \widetilde{x}}\chi _{Q}(x)$ and $G_{\vec{l}'',R_{Q,s}}(y)=e^{2\pi i\vec{l}''\cdot \widetilde{y}}
\chi_{R_{Q,s}}(y)$. Then, we get
\begin{eqnarray*}
K(x-y)&=&\sum_{Q\in\mathscr{D}}\sum_{s=1}^{M}K_{j}(x-y)\chi_{(Q\times R_{Q, s})}(x,y)
\\&=&\sum_{Q\in\mathscr{D}}\sum_{s=1}^{M}\sum_{\vec{l}\in\mathbb{Z}^{2n}}\lambda^j_{\vec{l},Q}\frac{1}{|Q|^{1/2}}F_{\vec{l}',Q}(x)\frac{1}{|R_{Q,s}|^{1/2}}G_{\vec{l}'',R_{Q,s}}(y).
\end{eqnarray*}
Thus, the kernel of $w^{\frac{1}{2}}[b, R_{j}] (w^{-\frac{1}{2}})$ can be represented as
\begin{align*}
K^{w}_{b}(x,y)=\sum_{Q\in\mathscr{D}}\sum_{s=1}^{M}\sum_{\vec{l}\in\mathbb{Z}^{2n}}
(b(x)-b(y))\lambda^j_{\vec{l},Q}\frac{1}{|Q|^{1/2}}
w^{\frac{1}{2}}(x)F_{\vec{l}',Q}(x)\frac{1}{|R_{Q,s}|^{1/2}}
G_{\vec{l}'',R_{Q,s}}(y)w^{-\frac{1}{2}}(y).
\end{align*}
For each $Q$ we rewrite $b(x)-b(y)$ as $(b(x)-\langle b\rangle_{Q})+(\langle b\rangle_{Q}-b(y))$. We get
\begin{align*} K^{w}_{b}(x,y)&=C_{[w]_{A_2}}\sum_{Q\in\mathscr{D}}\sum_{s=1}^{M}\sum_{\vec{l}\in\mathbb{Z}^{2n}} \sum_{m=0}^{1} { \lambda^j_{\vec{l},Q}}\frac{\left(b(x)-b_{Q}\right)^{m}w^{\frac{1}{2}}(x) F_{\vec{l}',Q}(x)}
{[w(Q)]^{\frac{1}{2}}}\\&\qquad\times
 \frac{\left(b_{Q}-b(y)\right)^{1-m}w^{-\frac{1}{2}}(y) G_{\vec{l}'',R_{Q,s}}(y)}{{[w^{-1}(Q)]^{\frac{1}{2}}}}.
\end{align*}
There is a large constant $K>1$ such that $KQ$ contains $Q\cup R_{Q,s}$.
We use the following notation 
\begin{align*}
\operatorname{osc}_\alpha(b,Q)=\bigg[|Q|^{-1} \int_{KQ}|b(u)-\langle b\rangle_{Q}|^{\alpha} du\bigg]^{1/\alpha},
\end{align*}
where $\alpha>\frac{2(1+\sigma_{w})}{\sigma_{w}}$ with $\sigma_w$ the reverse doubling index.  Then
$$
F_{\vec{l}',Q,m}(x)=\left(\operatorname{osc}_\alpha(b,Q)\right)^{-m}\frac{\left(b(x)-\langle b\rangle_{Q}\right)^{m}w^{\frac{1}{2}}(x)F_{\vec{l}',Q}(x)}{[w(Q)]^{\frac{1}{2}}}
$$
and
$$
G_{\vec{l}'',R_{Q,s},m}(y)= \left(\operatorname{osc}_\alpha(b,Q)\right)^{-(1-m)}\frac{\left(\langle b\rangle_{Q}-b(y)\right)^{1-m}w^{-\frac{1}{2}}(y) G_{\vec{l}'',R_{Q,s}}(y)}{{[w^{-1}(Q)]^{\frac{1}{2}}}}.
$$
For $m=0,1$, let $\beta=\frac{2(1+\sigma_{w})\alpha}{\alpha+2(1+\sigma_{w})}$ and $p=\frac{\alpha}{\beta}$. It is clear that $\beta>2$, $p>1$, and $\beta p'=2(1+\sigma_{w})$. By the H\"{o}lder inequality and Lemma \ref{reverse}, and $\operatorname{supp}\left(F_{\vec{l}',Q}\right) \subset Q,~\left\|F_{\vec{l}',Q}\right\|_{\infty} \leqslant 1$ yield that
\begin{align*}
~\left\|F_{\vec{l}',Q,m}\right\|_{L^\beta(\mathbb{R})}
&\leq \left(\int_{\mathbb{R}^n}\bigg|\left(\operatorname{osc}_\alpha(b,Q)\right)^{-m}\frac{\left(b(x)-\langle b\rangle_{Q}\right)^{m}w^{\frac{1}{2}}(x) F_{\vec{l}',Q}(x)}{{[w(Q)]^{\frac{1}{2}}}}\bigg|^{\beta}dx\right)^{\frac{1}{\beta}}\\
&\leq [w(Q)]^{-\frac{1}{2}}(\operatorname{osc}_\alpha(b,Q))^{-m}\bigg(\frac{1}{|Q|}\int_{Q}|
b(x)-\langle b\rangle_{Q}|^{\beta m}w^{\frac{\beta}{2}}(x)dx\bigg)^{\frac{1}{\beta}}|Q|^{\frac{1}{\beta}}\\
&\leq [w(Q)]^{-\frac{1}{2}}(\operatorname{osc}_\alpha(b,Q))^{-m}\bigg(\frac{1}{|Q|}\int_{Q}|
b(x)-\langle b\rangle_{Q}|^{p\beta m}dx\bigg)^{\frac{1}{p\beta}}\bigg( \frac{1}{|Q|}\int_{Q} w^{\frac{p'\beta}{2}}(x)dx\bigg)^{\frac{1}{p'\beta}}|Q|^{\frac{1}{\beta}}\\
&\leq [w(Q)]^{-\frac{1}{2}}(\operatorname{osc}_\alpha(b,Q))^{-m}\bigg(\frac{1}{|Q|}\int_{KQ}|
b(x)-\langle b\rangle_{Q}|^{\alpha}dx\bigg)^{\frac{m}{\alpha}}\bigg( \frac{1}{|Q|}\int_{Q} w^{1+\sigma_w}(x)dx\bigg)^{\frac{1}{2(1+\sigma_w)}}|Q|^{\frac{1}{\beta}}\\
&\leq C_{[w]_{A_2}}|Q|^{\frac{1}{\beta}-\frac{1}{2}}.
\end{align*}
We now consider $G_{\vec{l}'',R_{Q,s},m}$ and use a method similar to the proof above. Let $\beta=\frac{2(1+\sigma_{w})\alpha}{\alpha+2(1+\sigma_{w})}$ and $p=\frac{\alpha}{\beta}$, we have
\begin{align*}
&\left\|G_{\vec{l}'',R_{Q,s},m}\right\|_{L^{\beta}(\mathbb{R})}
\\&\leq \left(\int_{\mathbb{R}^n}\bigg|\left(\operatorname{osc}_\alpha(b,Q)\right)^{-(1-m)}\frac{\left(\langle b\rangle_{Q}-b(y)\right)^{1-m}w^{-\frac{1}{2}}(y)G_{\vec{l}'',R_{Q,s}}(y)}{[w^{-1}(Q)]^{\frac{1}{2}}}\bigg|^{\beta}dy\right)^{\frac{1}{\beta}}\\
&\leq [w^{-1}(Q)]^{-\frac{1}{2}}(\operatorname{osc}_\alpha(b,Q))^{-(1-m)}\bigg(\int_{R_{Q,s}}|
b(x)-\langle b\rangle_{Q}|^{\beta(1-m)}w^{-\frac{\beta}{2}}(y)dx\bigg)^{\frac{1}{\beta}}\\
&\leq [w^{-1}(Q)]^{-\frac{1}{2}}|Q|^{\frac{1}{\beta}}(\operatorname{osc}_\alpha(b,Q))^{-(1-m)}\bigg(\frac{1}{|Q|}\int_{R_{Q,s}}|
b(x)-\langle b\rangle_{Q}|^{p(1-m)\beta}dx\bigg)^{\frac{1}{p\beta}}\bigg(\frac{1}{|Q|}\int_{{Q}}
w^{-\frac{p'\beta}{2}}(y)dy\bigg)^{\frac{1}{p'\beta}}\\
&\leq [w^{-1}(Q)]^{-\frac{1}{2}}|Q|^{\frac{1}{\beta}}(\operatorname{osc}_\alpha(b,Q))^{-(1-m)}\bigg(\frac{1}{|Q|}\int_{KQ}|
b(x)-\langle b\rangle_{Q}|^{\alpha}dx\bigg)^{\frac{1-m}{\alpha}}\bigg(\frac{1}{|Q|}\int_{{Q}}
w^{-(1+\sigma_w)}(y)dy\bigg)^{\frac{1}{2(1+\sigma_{w})}}\\
&\leq [w^{-1}(Q)]^{-\frac{1}{2}}|Q|^{\frac{1}{\beta}}\bigg(\frac{K}{|KQ|}\int_{{KQ}}
w^{-(1+\sigma_w)}(y)dy\bigg)^{\frac{1}{2(1+\sigma_{w})}}\\
&\leq [w^{-1}(Q)]^{-\frac{1}{2}}|Q|^{\frac{1}{\beta}}K^{\frac{1}{2(1+\sigma_{w})}}
\left(\frac{w^{-1}(KQ)}{|KQ|}\right)^{\frac{1}{2}}\\
&\leq C({[w]_{A_2}},K)|Q|^{\frac{1}{\beta}-\frac{1}{2}},
\end{align*} where the last inequality comes from Lemma \ref{double}.
Therefore,
$$
w^{\frac{1}{2}}[b, R_{j}] (w^{-\frac{1}{2}})=C_{[w]_{A_2}} \sum_{\vec{l} \in \mathbb{Z}^{2n}} \sum_{m=0}^{1}\sum_{s=1}^{M}\sum_{Q \in \mathscr{D}} \lambda_{Q, \vec{l}}\operatorname{osc}_\alpha(b,Q)\left\langle f, G_{\vec{l}'',R_{Q,s},m}\right\rangle F_{\vec{l}',Q,m},
$$
where $\{G_{Q, \vec{l}'',m}\}$ and $\{F_{Q, \vec{l}',m}\}$ are NWO sequences and numbers $\left\{\lambda_{Q, \vec{l}}\right\}$ which satisfy $\left|\lambda_{Q, \vec{l}}\right| \lesssim \frac{1}{(1+|\vec{l}|)^{|\alpha|+|\beta|}}$ for all multi-index $\alpha, ~\beta\in\mathbb{Z}_{+}^{n}$. Thus, by Lemma \ref{eq-NWO1}, we have
\begin{align*}
\|w^{\frac{1}{2}}[b, R_{j}] (w^{-\frac{1}{2}})\|_{ S^{n,\infty}\left(L^{2}(\mathbb{R}^n)\right)}\lesssim \|\operatorname{osc}_\alpha(b,Q)\|_{\ell^{n,\infty}}.
\end{align*}
 By \cite[Theorem 1 and Remark (d)]{F2022} (see also \cite{RS1989}), we know that $\operatorname{osc}_\alpha(b,Q)\in \ell^{n,\infty}$ follows from $b\in \dot{W}^{1,n}(\mathbb{R}^{n})$, Then $ w^{\frac{1}{2}}[b, R_{j}] (w^{-\frac{1}{2}})\in S^{n,\infty}\left(L^{2}(\mathbb{R}^n)\right)$. Hence, we are done with the proof of the sufficient condition in Theorem \ref{Athm2}.

\subsection{Proof of the necessary condition}
In this subsection, we use the idea of \cite{RS1989}. We assume that $[b, R_{j}]\in S^{n,\infty}\left(L^{2}(w)\right)$, then prove that $b\in \dot{W}^{1,n}(\mathbb{R}^n)$.

 First, choosing two cubes $Q$ and $\hat{Q}$ in $\mathscr{D}$, as Lemma \ref{lemA4}. Define
\begin{align*}
J_{Q}(x,y)=|Q|^{-2}K^{-1}_{j}(x-y)\chi_{Q}(x)\chi_{\hat{Q}}(y).
\end{align*}
For $K_{j}(x-y)^{-1}$, applying the multiple Fourier series on $Q\times\hat{Q}$, we can write
\begin{align*}
K^{-1}_{j}(x-y)=\sum_{\vec{l}\in\mathbb{Z}^{2n}}c^{j}_{\vec{l}}e^{2\pi i\vec{l}'\cdot \widetilde{x}}e^{2\pi i\vec{l}''\cdot \widetilde{y}}\chi
_{Q}(x)\cdot \chi _{\hat{Q}}(y),
\end{align*}
where $x_{i}=C_{Q}^{(i)}+\ell(Q)\tilde{x}_{i}$, $y_{i}=C_{\hat{Q}}^{(i)}+\ell(\hat{Q})\tilde{y}_{i}$, $i=1,2,\cdot\cdot\cdot,n$, and $\vec{l}=(\vec{l}', \vec{l}'')$ where $\vec{l}'=(l_1,l_2,\cdot\cdot\cdot,l_n)$, $\vec{l}''=(l_{n+1},l_{n+2},\cdot\cdot\cdot,l_{2n})$, and
\[
c^j_{\vec{l}}=\int_{\hat{Q}}\int_{Q}K^{-1}_{j}(x-y)\chi_{(Q\times \hat{Q})}e^{-2\pi i\vec{l}'\cdot \widetilde{x}
}e^{-2\pi i\vec{l}''\cdot \widetilde{y}}dxdy\frac{1}{|Q|}\frac{1}{|\hat{Q}|}.
\]
Similar to the estimate in the previous section, using $$|\partial^{\alpha}_{x}\partial^{\beta}_{y}K^{-1}_{j}(x-y)|\leq C(\alpha,\beta){|x-y|^{n-|\alpha|-|\beta|}},$$
$|Q|=|\hat{Q}|$ and $\mathrm{distance}(Q,\hat{Q})\approx |Q|$ yields
 \begin{align*}
 |c^j_{\vec{l}}|&\lesssim \frac{1}{(1+|\vec{l}|)^{|\alpha|+|\beta|}}\ell(Q)^{|\alpha|}\ell(\hat{Q})^{|\beta|}\int_{\hat{Q}}\int_{Q} |\partial^{\alpha}_{x}\partial^{\beta}_{y}K^{-1}_{j}(x-y)|dxdy\frac{1}{|Q|}\frac{1}{|\hat{Q}|}\\
 &\lesssim \frac{1}{(1+|\vec{l}|)^{|\alpha|+|\beta|}}\ell(Q)^{|\alpha|}\ell(\hat{Q})^{|\beta|}\int_{\hat{Q}}\int_{Q} {|x-y|^{n-|\alpha|-|\beta|}}dxdy\frac{1}{|Q|}\frac{1}{|\hat{Q}|}\\
&\lesssim{|Q|}\frac{1}{(1+|\vec{l}|)^{|\alpha|+|\beta|}}.
\end{align*}
where $\alpha,~\beta\in \mathbb{Z}_+^{n}$ are multi-index.
Therefore, we can denote $\lambda^j_{\vec{l},Q}=\frac{1}{|Q|}c^j_{\vec{l}}$, then $$|\lambda^j_{\vec{l},Q}|\lesssim \frac{1}{(1+|l|)^{|\alpha|+|\beta|}}.$$ Obviously,
\begin{eqnarray*}
J_{Q}(x,y)=\sum_{\vec{l}\in\mathbb{Z}^{2n}}\lambda^j_{\vec{l},Q}\frac{1}{|Q|^{1/2}}
F_{\vec{l}',Q}(x)\frac{1}{|\hat{Q}|^{1/2}}G_{\vec{l}'',\hat{Q}}(y),
\end{eqnarray*}
where
$F_{\vec{l}',Q}(x)=e^{2\pi i\vec{l}'\cdot \widetilde{x}}\chi _{Q}(x)$ and $G_{\vec{l}'',\hat{Q}}(y)=e^{2\pi i\vec{l}''\cdot \widetilde{y}}
\chi_{\hat{Q}}(y)$.

Next, set an arbitrary function { $\varepsilon:\mathbb{R}^n\rightarrow \{-1,1\}$}. Define the operator $L_{Q}$ as
$$w^{\frac{1}{2}}(x)L_{Q}(w^{-\frac{1}{2}}f)(x)=\int_{\mathbb{R}^n}w^{\frac{1}{2}}(x)
\varepsilon_{Q}(x)J_{Q}(x,y)w^{-\frac{1}{2}}(y)f(y)dy,$$
{ where supp$\varepsilon_{Q}\subset Q$}.
Considering an arbitrary sequence $\{a_{Q}\}_{Q\in\mathscr{D}}\in \ell^{\frac{n}{n-1},1}$. Here $\ell^{\frac{n}{n-1},1}$ is the Lorentz sequence space defined as
the set of all sequences $\{a_{Q}\}_{Q\in\mathscr{D}}$ such that
$$ \big\|\{a_{Q}\}_{Q\in\mathscr{D}}\big\|_{\ell^{\frac{n}{n-1},1}} =\sum_{k=1}^\infty k^{\frac{n}{n-1}  -1}a^*_k, $$
where the sequence $\{a^*_k\}$ is  the sequence $\{ |a_Q|\}$ rearranged in a decreasing order.

Define the operator $L$ as
\begin{align*}w^{\frac{1}{2}}(x)L(w^{-\frac{1}{2}}f)(x)
&=\sum_{Q\in\mathscr{D}}a_{Q}
w^{\frac{1}{2}}(x)L_{Q}(w^{-\frac{1}{2}}f)(x).
\end{align*}
Therefore, we also write
\begin{align*}
w^{\frac{1}{2}}(x)L(w^{-\frac{1}{2}}f)(x)=C_{[w]_{A_2}}
\sum_{Q\in\mathscr{D}}\sum_{\vec{l}\in\mathbb{Z}^{2n}}\lambda^j_{\vec{l},Q}a_{Q}\langle f,\tilde{G}_{\vec{l}'',\hat{Q}}   \rangle \tilde{F}_{\vec{l}',Q}(x),
\end{align*}
where
$$   \tilde{G}_{\vec{l}'',\hat{Q}}(y)=\frac{G_{\vec{l}'',\hat{Q}}(y)w^{-\frac{1}{2}}(y)}{(w^{-1}(Q))^{\frac{1}{2}}}    \quad\text{and}\quad   \tilde{F}_{\vec{l}',Q}(x)= \frac{F_{\vec{l}',Q}(x)w^{\frac{1}{2}}(x)}{(w(Q))^{\frac{1}{2}}}.$$
By Lemma \ref{reverse}, similar to the proof of Section 3, they are NWO sequences. Thus, applying { Lemma \ref{WSweight}} to give
\begin{align*}
\|L\|_{S^{\frac{n}{n-1},1}(L^{2}(w))}
&=\|w^{\frac{1}{2}}Lw^{-\frac{1}{2}}\|_{S^{\frac{n}{n-1},1}(L^{2}(\mathbb{R}^{n}))}\leq\|a_{Q}\|_{\ell^{\frac{n}{n-1},1}}.
\end{align*}

{ Using the idea of \cite [p.262]{RS1989}}, we also can obtain
\begin{align*}
\operatorname{Trace}(w^{\frac{1}{2}}[b,R_j]L_{Q}(w^{-\frac{1}{2}}))=|Q|^{-2}\int_{Q}\int_{\hat{Q}}(b(x)-b(y))\varepsilon_{Q}(x)dydx.
\end{align*}
Note that $|b-\langle b\rangle_{Q}|\leq |b-\langle b\rangle_{\hat{Q}}|+\frac{1}{|Q|}\int_{Q}|b(x)-\langle b\rangle_{\hat{Q}}|dx$, and by the definition of $\varepsilon_{Q}$, we have
\begin{align*}
\operatorname{Trace}(w^{\frac{1}{2}}[b,R_j]L_{Q}(w^{\frac{1}{2}}))&\gtrsim \frac{1}{|Q|}\int_{Q}|b(x)-\langle b\rangle_{\hat{Q}}|dx\\
&\gtrsim\frac{1}{|Q|}\int_{Q}|b(x)-\langle b\rangle_{Q}|dx\\
&=:M(b,Q).
\end{align*}
Therefore, by duality, there exists a sequence $\{a_Q\}_{Q\in\mathscr{D}}$ with $\|a_{Q}\|_{\ell^{\frac{n}{n-1},1}}\leq1$ such that
\begin{align*}
\|b\|_{\dot{W}^{1,n}(\mathbb{R}^n)}&\lesssim\|M(b,Q)\|_{\ell^{n,\infty}}\\
&\lesssim\|\operatorname{Trace}(w^{\frac{1}{2}}[b,R_j]L_{Q}(w^{-\frac{1}{2}}))\|_{\ell^{n,\infty}}\\
&=\sup_{\|a_{Q}\|_{\ell^{\frac{n}{n-1},1}}\leq1}\sum_{Q\in \mathscr{D}}\operatorname{Trace}(w^{\frac{1}{2}}[b,R_j]L_{Q}(w^{-\frac{1}{2}}))\cdot a_{Q}\\
&=\sup_{\|a_{Q}\|_{\ell^{\frac{n}{n-1},1}}\leq1}\operatorname{Trace}(w^{\frac{1}{2}}[b,R_j]L(w^{-\frac{1}{2}}))\\
&\lesssim\sup_{\|a_{Q}\|_{\ell^{\frac{n}{n-1},1}}\leq1}\|w^{\frac{1}{2}}[b,R_j](w^{-\frac{1}{2}})\|_{S^{n,\infty}(L^{2}(\mathbb{R}^{n}))}\|w^{\frac{1}{2}}L(w^{-\frac{1}{2}})\|_{S^{\frac{n}{n-1},1}(L^{2}(\mathbb{R}^{n}))}
\\&\lesssim\|[b,R_j]\|_{S^{n,\infty}(L^{2}(w))},
\end{align*}
where the first inequality comes from \cite[Theorem 1 and Remark (d)]{F2022} (see also \cite{RS1989}).
Hence, the proof of the necessary condition in Theorem \ref{Athm2} is complete.

{ \section{Further Discussions on one dimensional setting}\label{one dimensional}
In Theorem \ref{Athm1} and Theorem \ref{Athm2}, we mainly consider the commutator with Riersz transform $R_j$ $(j=1,2,\cdots,n)$ in higher dimensions on $\mathbb R^{n}$. 
When $n=1$, Peller \cite{P2003} obtained the follow result in the unweighted case,
\begin{Theorem}[\cite{P2003}] For $b\in {\rm VMO}(\mathbb{R})$, and $0<p<\infty$, we have 
	\begin{align*}
		\|[b,H]\|_{S^{p}(L^2(\mathbb{R}))}\approx\|b\|_{B_{1/p}^{p,p}(\mathbb{R})}.
	\end{align*}
\end{Theorem}
For $p=2$, Lacey and the last two authors in \cite{LLW2022} considered that Schatten classes and the commutator $[b,H]$ in the two weight setting(see also Theorem A in Section \ref{Introduction}). In \cite[Section 7]{LLW2022}, the authors raised two questions about one weight in one dimension.

{\rm(i)} For $b\in {\rm VMO}(\mathbb{R})$, and $1<p<\infty$. Is
$
	\|[b,H]\|_{S^{p}(L^2(w))}\approx\|b\|_{B_{1/p}^{p,p}(\mathbb{R})}
$ established?

{\rm(ii)} Can the above conclusion be extended to $0<p\le1$?
\\
Similar to the proof of {\rm(1)} in Theorem \ref{Athm1}, we can give a positive answer to problem {\rm(i)}. However, we can't come up with a good way to solve problem {\rm(ii)} in this paper. So, this remains an open problem.

}

\section{An Application: The Quantised Derivative}\label{application}

Let $n>1$ be an integer, and let $x_1, x_2, \ldots, x_n$ be the coordinates of $\mathbb{R}^n$.
%
%
For $j=1, \ldots, n$, we define $D_j$ to be the derivative in the direction $x_j$,
$$
D_j=\frac{1}{i} \frac{\partial}{\partial x_j}=-i \partial_j.
$$
When $f \in L^{\infty}\left(\mathbb{R}^n\right)$ is not a smooth function then $D_j f$ denotes the distributional derivative of $f$.
We also consider $D_j$ as a self-adjoint operator on $L^2\left(\mathbb{R}^n\right)$ with its standard domain of square integrable functions with a square integrable weak derivative in the direction $x_j$.
This is equivalent to the closure of the symmetric operator $D_j$ restricted to Schwartz functions.
We use the notation $\nabla f=i\left(D_1 f, D_2 f, \ldots, D_n f\right)$ for an essentially bounded function $f \in L^{\infty}\left(\mathbb{R}^n\right)$.
For a square integrable function $f$ with a square integrable derivative in each direction we consider $\nabla$ as an unbounded operator from $L^2\left(\mathbb{R}^n\right)$ to the Bochner space $L^2\left(\mathbb{R}^n, \mathbb{C}^n\right)$.

Let $N=2^{\lfloor n / 2\rfloor}$.
We use $n$-dimensional Euclidean gamma matrices, which are $N \times N$ self-adjoint complex matrices $\gamma_1, \ldots, \gamma_n$ satisfying the anticommutation relation,
$$
\gamma_j \gamma_k+\gamma_k \gamma_j=2 \delta_{j, k}, \quad 1 \leq j, k \leq n,
$$
where $\delta$ is the Kronecker delta.
The precise choice of matrices satisfying this relation is unimportant so we assume that a choice is fixed for the rest of this paper.

Using this choice of gamma matrices, we can define the $n$-dimensional Dirac operator,
$$
\mathcal{D}=\sum_{j=1}^n \gamma_j \otimes D_j.
$$
This is a linear operator on the Hilbert space $\mathbb{C}^N \otimes L^2\left(\mathbb{R}^n\right)$ initially defined with dense domain $\mathbb{C}^N \otimes \mathcal{S}\left(\mathbb{R}^n\right)$, where $\mathcal{S}\left(\mathbb{R}^n\right)$ is the Schwartz space of functions on $\mathbb{R}^n$.
It is easily seen that $\mathcal{D}$ is symmetric on this domain. Taking the closure we obtain a self-adjoint operator which we also denote by $\mathcal{D}$.
We then define the $\operatorname{sign}$ of $\mathcal{D}$ as the operator $\operatorname{sgn}(\mathcal{D})$ via the Borel functional calculus, i.e.,
$\operatorname{sgn}(\mathcal{D}) = {\mathcal{D}\over |\mathcal{D}|}.$

Given $f \in L^{\infty}\left(\mathbb{R}^n\right)$, denote by $M_f$ the operator of pointwise multiplication by $f$ on the Hilbert space $L^2\left(\mathbb{R}^n\right)$.
The operator $1 \otimes M_f$ is a bounded linear operator on $\mathbb{C}^N \otimes L^2\left(\mathbb{R}^n\right)$, where $1$ denotes the identity operator on $\mathbb{C}^N$. The commutator,
$$
\bar{d} f:=i\left[\operatorname{sgn}(\mathcal{D}), 1 \otimes M_f\right]
$$
denotes the quantised derivative of Alain Connes introduced in \cite[$\mathrm{IV}$]{C1994}.
It is of particular interest in the quantised calculus to determine conditions on $f$ such that $\bar{d} f \in S^{n, \infty}(\mathbb{C}^N \otimes L^2(\mathbb{R}^n) )$.
The asymptotic behaviour of the singular values of the quantised derivative denote the dimension of the infinitesimal in the quantised calculus.
That the sequence of singular values belongs to the weak space $\ell^{n, \infty}$ when the dimension of the Euclidean space is $n$ indicates analogous behaviour between quantum derivatives and differential forms.
Specifically, a product of $n$ derivatives lies in the space $S^{1, \infty}(\mathbb{C}^N \otimes L^2\left(\mathbb{R}^n\right))$, which is the only weak space admitting a non-trivial trace that acts as the integral.

In one dimension, necessary and sufficient conditions on $f \in L^{\infty}(\mathbb{R})$ such that $\left[\operatorname{sgn}(-i \frac{d}{d x}), M_f\right] \in S^{p, q}(\mathbb{C}^N \otimes L^2\left(\mathbb{R}\right))$ where $p, q \in(0, \infty]$ are provided by Peller in \cite[Chapter 4 , Theorem 4.4]{P2003}.  Janson and Wolff \cite{JW1982}, and Connes, Sullivan and Teleman \cite{CST1994} have studied necessary and sufficient conditions for $\bar{d} f \in S^{p, q}(\mathbb{C}^N \otimes L^2\left(\mathbb{R}^n\right))$ with $p, q \in(0, \infty]$ in the higher dimensional case $n>1$.  The case of $p=q$ was studied by Janson and Wolff in their paper \cite{JW1982}.  They proved that when $p>n$, a necessary and sufficient condition for $\bar{d} f\in S^p(\mathbb{C}^N \otimes L^2\left(\mathbb{R}^n\right))$ is that $f$ is in the Besov space $B^{p, p}_{n / p}\left(\mathbb{R}^n\right)$.
They also show that if $p \leq n$, then $\bar{d} f \in S^p$ if and only if $f$ is a constant.

The case of $p \neq q$ with $p \in[1, \infty)$ and $q \in[1, \infty]$ was answered by Rochberg and Semmes in \cite[Corollary 2.8, Theorem 3.4]{RS1989}.
Necessary and sufficient conditions on $f \in$ $L^{\infty}\left(\mathbb{R}^n\right)$ are given so that $\bar{d} f \in S^{p, q}(\mathbb{C}^N \otimes L^2\left(\mathbb{R}^n\right))$.
These conditions are given in terms of the mean oscillation of $f$, and it is not obvious whether an equivalent condition could be given in terms of more familiar function spaces.
In the Appendix of Connes, Sullivan and Teleman's paper
\cite[p. 679]{CST1994}, it is proved that necessary and sufficient conditions for $\bar{d} f \in S^{n, \infty}(\mathbb{C}^N \otimes L^2\left(\mathbb{R}^n\right))$ are that $f \in L_{\mathrm{loc}}^1\left(\mathbb{R}^n\right)$ and $\nabla f \in L^n\left(\mathbb{R}^n, \mathbb{C}^n\right)$.

Recently, Lord--McDonald--Sukochev--Zanin \cite{LMSZ2017} gave a complete and different proof of this result under the assumption that $f \in L^{\infty}\left(\mathbb{R}^n\right)$ using double operator integrals.
Their method  gave sharp bounds on the quasinorm $\|\bar{d} f\|_{S^{n,\infty}(\mathbb{C}^N \otimes L^2\left(\mathbb{R}^n\right))}$. For the norm $\nabla f \in L^n\left(\mathbb{R}^n, \mathbb{C}^n\right)$, they implicitly assumed that the essentially bounded function $f$ has weak partial derivatives and that the Bochner norm of $\nabla f$ in $L^n\left(\mathbb{R}^n, \mathbb{C}^n\right)$,
$$
\|\nabla f\|_{L^n\left(\mathbb{R}^n, \mathbb{C}^n\right)}=\left(\int_{\mathbb{R}^n}\|(\nabla f)(x)\|_n^n d x\right)^{1 / n}=\left(\int_{\mathbb{R}^n} \sum_{j=1}^n\left|D_j f(x)\right|^n d x\right)^{1 / n},
$$
is finite.  The key step that they established is a new trace formula described as follows, which is analogous to
Connes in \cite{C1988}.

Recall that a trace on $S^{1, \infty}(\mathbb{C}^N \otimes L^2\left(\mathbb{R}^n\right))$ is a linear functional $\varphi: S^{1, \infty}(\mathbb{C}^N \otimes L^2\left(\mathbb{R}^n\right)) \rightarrow \mathbb{C}$ such that $\varphi([A, B])=0$ for all bounded operators $A$ and for all $B \in S^{1, \infty}(\mathbb{C}^N \otimes L^2\left(\mathbb{R}^n\right))$. The trace $\varphi$ is called continuous when it is continuous with respect to the $S^{1, \infty}(\mathbb{C}^N \otimes L^2\left(\mathbb{R}^n\right))$ quasinorm.
Given an orthonormal basis $\left\{e_n\right\}_{n=0}^{\infty}$ of $H$, define the operator $T:=\operatorname{diag}\left\{\frac{1}{n+1}\right\}_{n=0}^{\infty}$ by $\left\langle e_n, T e_m\right\rangle=\delta_{n, m} \frac{1}{n+1}$. The linear functional
$\varphi$ is called normalised when
$$
\varphi\left(\operatorname{diag}\left\{\frac{1}{n+1}\right\}_{n=0}^{\infty}\right)=1.
$$
The property that $\varphi$ is normalised is independent of the choice of orthonormal basis, since for all unitary operators $U$ and all bounded operators $B$ we have $\varphi\left(U B U^*\right)=\varphi(B)$.
\begin{pro}[\cite{LMSZ2017}]\label{Sthm6.1} Let $f \in L^{\infty}\left(\mathbb{R}^n\right)$ be real valued and such that $\nabla f \in L^n\left(\mathbb{R}^n, \mathbb{C}^n\right)$. Then there is a constant $c_n>0$ such that for any continuous normalised trace $\varphi$ on $S^{1, \infty}(\mathbb{C}^N \otimes L^2\left(\mathbb{R}^n\right))$ we have
$$
\varphi\left(|\bar{d} f|^n\right)=c_n \int_{\mathbb{R}^n}\|\nabla f(x)\|_2^n d x .
$$
\end{pro}
Proposition \ref{Sthm6.1} is the analogue of \cite[Theorem 3(3)]{C1988} for functions on the non-compact manifold $\mathbb{R}^n$. It is  also stated for a larger class of functions than \cite[Theorem 3(3)]{C1988} which is proved for smooth functions.  Based on this trace formula, in \cite{LMSZ2017} they obtained that
\begin{pro}[\cite{LMSZ2017}]\label{Sthm6.2} Let $n>1$ and $f \in L^{\infty}\left(\mathbb{R}^n\right)$. Then, for $\bar{d} f\in S^{n,\infty}(\mathbb{C}^N \otimes L^2\left(\mathbb{R}^n\right))$, it is necessary and sufficient that
$\nabla f\in L^n\left(\mathbb{R}^n, \mathbb{C}^n\right)$. Further, there exist positive constants $c$ and $C$ depending only on $n$ such that,
$$
c\|\nabla f\|_{L^n\left(\mathbb{R}^n, \mathbb{C}^n\right)}\leq \|\bar{d} f\|_{S^{n,\infty}(\mathbb{C}^N \otimes L^2\left(\mathbb{R}^n\right))}\leq C\|\nabla f\|_{L^n\left(\mathbb{R}^n, \mathbb{C}^n\right)}.
$$
\end{pro}

From our Theorem \ref{Athm2}, we have the following result in this direction:

\begin{thm}\label{Athm3}
Suppose $n>1$,  $f\in {\rm VMO}(\mathbb{R}^n)$, $w\in A_{2}$. Then $\bar{d} f\in S^{n,\infty}(\mathbb{C}^N \otimes L^2\left(w\right))$ if and only if $f\in \dot{W}^{1,n}(\mathbb{R}^n)$. Moreover,
$$
 \|\bar{d} f\|_{S^{n,\infty}(\mathbb{C}^N \otimes L^2\left(w\right))}\approx \|f\|_{ \dot{W}^{1,n}(\mathbb{R}^n)}.
$$
\end{thm}
\begin{proof}
For the convenience of the readers, we provide the details of the link between $\bar{d} f$ and $[f,\nabla \Delta^{-1/2}]$, where $\Delta$ is the standard Laplacian on $\mathbb R^n$.
In fact, from the definition of $\mathcal{D}$ and the property of these self-adjoint complex matrices $\gamma_1, \ldots, \gamma_n$, we see that
$$
\mathcal{D}^2=-1\otimes \Delta.
$$
Moreover,
$\operatorname{sgn}(\mathcal{D})$ which can equivalently be expressed as follows.
$$
\operatorname{sgn}(\mathcal{D})=\sum_{j=1}^n \gamma_j \otimes D_j\Delta^{-1/2}=\sum_{j=1}^n \gamma_j \otimes R_j,
$$
where $R_j$ is the $j$th Riesz transform. Hence,
\begin{align*}
\bar{d} f&=i\left[\operatorname{sgn}(\mathcal{D}), 1 \otimes M_f\right] = i\left[\sum_{j=1}^n \gamma_j \otimes R_j, 1 \otimes M_f\right]
= i\sum_{j=1}^n\left[ \gamma_j \otimes R_j, 1 \otimes M_f\right]\\
&= i\sum_{j=1}^n(\gamma_j \otimes R_jM_f -\gamma_j \otimes M_fR_j)\\
&=i\sum_{j=1}^n\gamma_j \otimes [R_j,M_f].
\end{align*}
Thus, the result follows from Theorem \ref{Athm2}.
\end{proof}

\bigskip
\noindent{\bf Acknowledgements:} The authors really appreciate all the efforts of referees for reading and checking the paper, and  for the  valuable comments. 

JL would like to thank E. McDonald, F. Sukochev and D. Zanin for helpful discussions on the quantised derivative.

JL's research supported by Australian Research Council DP 220100285. BDW's research is supported in part by National Science Foundation Grants DMS \#1800057, \#2054863, and \#20000510 and Australian Research Council DP 220100285. ZG's research supported by the China Scholarship Council [grant number 202006460063]

\arraycolsep=1pt


\begin{thebibliography}{99}

\bibitem{B}
      S. Bloom, A commutator theorem and weighted BMO,
     Trans. Amer. Math. Soc., 292 (1985),
103--122.


  \bibitem{C1965}
A.-P. Calder\'{o}n, Commutators of singular integral operators, Proc. Natl. Acad. Sci. USA 53 (1965) 1092--1099.

  \bibitem{CDW2019}
P. Chen, X. Duong, J. Li and Q. Wu, Compactness of Riesz transform commutator on stratified Lie groups, J.
Funct. Anal., 277 (2019), no. 6, 1639--1676.

  \bibitem{CLMS}
R. Coifman, P.L. Lions, Y. Meyer and S. Semmes, Compensated compactness and Hardy sapces, J. Math. Pures Appl., 72 (1993), 247--286.

  \bibitem{CRW1976} R. Coifman, R. Rochberg and G. Weiss, Factorization theorems for Hardy spaces in several variables, { Ann. of Math}., { 103} (1976),  611--635.

  \bibitem{C1988}
A. Connes, The action functional in noncommutative geometry, Comm. Math. Phys. 117(4) (1988) 673--683.

\bibitem{C1994}
A. Connes, Noncommutative Geometry, Academic Press, Inc., San Diego, CA, 1994.

\bibitem{CST1994}
A. Connes, D. Sullivan, N. Teleman, Quasiconformal mappings, operators on Hilbert space, and local formulae for characteristic classes, Topology 33(4) (1994) 663-681.




  \bibitem{FLL2021}
Z. Fan, M. T. Lacey and J. Li, Schatten classes and commutators of Riesz Transform on Heisenberg group and applications, arXiv: 2107.10569v3.


  \bibitem{FLMSZ}
Z. Fan, J. Li, E. McDonald, F. Sukochev and D. Zanin,
Endpoint weak Schatten class estimates and trace formula for commutators of Riesz transforms with multipliers on Heisenberg groups, arXiv: 2201.12350.


  \bibitem{F2022}
R. L. Frank,
A characterization of $\dot{W}^{1,p}(\mathbb{R}^{n})$, arXiv: 2203.01001.

  \bibitem{FSZ2022}
R. L. Frank, F. Sukochev and D. Zanin,
Asymptotics of singular values for quantum derivatives, arXiv:2209.12559.

  \bibitem{GG2017}
 H. Gimperlein and M. Goffeng. Nonclassical spectral asymptotics and Dixmier traces: from circles to contact manifolds. Forum Math. Sigma,
5: Paper No. e3, 57, 2017.


  \bibitem{G2014}
L. Grafakos, Modern fourier analysis. third edition., Springer, 2014.

  \bibitem{HLW2017}
I. Holmes, M. T. Lacey, and B. D. Wick, Commutators in the two-weight setting. Math. Ann. 367 (2017), no. 1-2, 51--80.

\bibitem{Hy}
{  T. Hyt\"{o}nen. The $L^p$-to-$L^q$ boundedness of commutators with applications to the Jacobian operator, J. Math. Pures Appl., 156 (2021),  351--391.}
  
\bibitem{HK2012}
{ T. Hyt\"{o}nen, A. Kairema, 
Systems of dyadic cubes in a doubling metric space, 
Colloq. Math. 126 (2012), 1--33.
}
\bibitem{JW1982}
S. Janson, T. H. Wolff, Schatten classes and commutators of singular integral operators, Ark. Mat. 20(2) (1982) 301--310.

\bibitem{LL2022} M. Lacey and J. Li, 
       Compactness of commutator of Riesz transforms in the two weight
  setting, J. Math. Anal. Appl., 508 (2022), 
 Paper No. 125869.

  \bibitem{LLW2022}
M. T. Lacey, J. Li and B. D. Wick, Schatten classes and commutator in the two weight setting. I. Hilbert transform, arXiv: 2202. 11854v1.

  \bibitem{LPPW2008}
M. T. Lacey, J. Pipher, S. Petermichl, and B. D. Wick, Iterated Riesz
Commutators: A Simple Proof of Boundedness, Proceedings of 8th Intl. Conf. on Harm.
Analysis and PDE at El Escorial, Madrid (Spain), 2008, available at http://www.arxiv.org/
abs/0808.0832.



  \bibitem{LMSZ2017}
S. Lord, E. McDonald, F. Sukochev and D. Zanin,
Quantum differentiability of essentially bounded functions on Euclidean space,
J. Funct. Anal.,  273 (2017) 2353--2387.


 \bibitem{N} Z. Nehari,
On bounded bilinear forms,
Ann. of Math., { 65} (1957), 153--162.


  \bibitem{P2003}
V. V. Peller, Hankel operators and their applications. Springer Monographs in Mathematics. Springer-Verlag, New York, 2003.

  \bibitem{Pe2008}
S. Petermichl, The sharp weighted bound for the Riesz transforms. Proc. Amer. Math. Soc.,
136(4)(2008), 1237-1249.

  \bibitem{PTV2002}
S. Petermichl, S. Treil and A. Volberg, Why are the Riesz transforms averages of the dyadic
shift?, Proceedings of the 6th international conference on harmonic analysis (El Escorial),
Publ. Mat. (2002), Extra Vol., pp. 209-228.


  \bibitem{RS1989}
R. Rochberg and S. Semmes, Nearly weakly orthonormal sequences, singular value estimates, and Calderon--
Zygmund operators, J. Funct. Anal., 86 (1989), 237--306.

 














\end{thebibliography}
\end{document}